\documentclass[11pt]{article}
\usepackage[margin=1in]{geometry}
\usepackage{graphicx,xcolor,cite,mathtools}
\usepackage{amssymb,amsmath,amsthm,mathrsfs,paralist,bm,esint,setspace}
\RequirePackage[colorlinks,citecolor=blue,urlcolor=blue]{hyperref}
\usepackage{marginnote}

\newcommand{\R}{\mathbb{R}}
\newcommand{\E}{\mathrm{E}}
\newcommand{\HH}{\mathcal{H}}

\renewcommand{\d}{\mathrm{d}}
\newcommand{\<}{\langle}
\renewcommand{\>}{\rangle}
\newcommand{\e}{\mathrm{e}}
\newcommand{\Var}{\text{\rm Var}}

\newcommand{\lip}{\text{\rm Lip}}

\DeclareMathOperator{\Cov}{\text{\rm Cov}}

\input cyracc.def

\newtheorem{stat}{Statement}[section]
\newtheorem{proposition}[stat]{Proposition}

\newtheorem{theorem}[stat]{Theorem}
\newtheorem{lemma}[stat]{Lemma}
\theoremstyle{definition}

\numberwithin{equation}{section}

\title{Spatial ergodicity and central limit theorems for parabolic Anderson model with delta initial condition
	}
\author{Le Chen\\Emory University\\\texttt{le.chen@emory.edu}\\
	\and
		Davar Khoshnevisan\\University of Utah\\\texttt{davar@math.utah.edu}\\
	\and\and
		David Nualart\\University of Kansas\\\texttt{nualart@ku.edu}\\
	\and
		Fei Pu\\Beijing Normal University\\\texttt{pufeibnu@gmail.com}
	}
\date{\today}

\begin{document}
\maketitle

\begin{abstract}
	Let $\{u(t\,, x)\}_{t >0, x \in\R}$ denote the solution to the
	parabolic Anderson model with  initial condition $\delta_0$ and driven by space-time
	white noise on $\R_+\times\R$, and let $\bm{p}_t(x):=
	(2\pi t)^{-1/2}\exp\{-x^2/(2t)\}$ denote the standard Gaussian
	heat kernel on the line. We use a non-trivial adaptation of
	the methods in our companion papers \cite{CKNP,CKNP_b} in order to prove that the random field
	$x\mapsto u(t\,,x)/\bm{p}_t(x)$ is  ergodic for every $t >0$. And we establish an associated
	quantitative central limit theorem following the approach based on
	the Malliavin-Stein method introduced in  Huang, Nualart, and Viitasaari \cite{HNV2018}.
\end{abstract}

\bigskip

 \smallskip

\smallskip

\noindent{\it Running head:} Ergodicity and CLT for PAM.


\section{Introduction}

Consider the {\em parabolic Anderson model},
\begin{equation}\label{PAM}
	\partial_t u(t\,, x) = \tfrac12\partial_x^2 u(t\,, x) + u(t\,, x)\eta(t \,, x),\qquad t >0, \, x \in \R,
\end{equation}
with delta initial condition $u(0) = \delta_0$, where $\eta$ denotes space-time white noise on
$\R_+\times\R$.
Following Walsh\cite{Walsh}, we interpret the stochastic PDE \eqref{PAM} in
the following mild form:
\begin{align}\label{mild}
	u(t\, , x) = \bm{p}_{t}(x) + \int_{(0,t)\times\R}\bm{p}_{t - s}(x - y)u(s\,, y)\,\eta(\d s\,\d y),
\end{align}
where
\[
	\bm{p}_{t}(x) = \frac{1}{\sqrt{2\pi t}}\,\e^{-x^2/(2t)}
	\qquad\text{for all }t > 0\text{ and }x \in \R.
\]

Consider the following renormalization of the solution to \eqref{PAM}:
\begin{equation} \label{U}
	U(t\,,x):= \frac{u(t\,,x)}{\bm{p}_t(x)}
	\qquad\text{for all $t>0$ and $x\in\R$}.
\end{equation}
It is not too hard to prove that $\lim_{t\downarrow0}U(t\,,x)=1$ in $L^k(\Omega)$
for all $x\in\R$ and $k\ge 2$; see Lemma \ref{lem:||U-1||} below.
Therefore, we also define
\[
	U(0\,,x) :=1\qquad\text{for all $x\in\R$},
\]
throughout.

Amir, Corwin, and Quastel \cite[Proposition 1.4]{ACQ11} have shown
that the process $U(t):=\{U(t\,,x)\}_{x\in\R}$ is stationary for every $t>0$.
The formulation \eqref{mild} of the stochastic PDE \eqref{PAM}
can be recast equivalently in terms of $U$ as follows:
\[
	U(t\,,x) = 1 + \int_{(0,t)\times\R}\frac{\bm{p}_{t - s}(x - y) \bm{p}_s(y)}{\bm{p}_t(x)}\,
	U(s\,,y)\,\eta(\d s\,\d y).
\]
Because
\begin{equation}\label{PPPP}
	\frac{\bm{p}_{t-s}(a)\bm{p}_s(b)}{\bm{p}_t(a+b)} =
	\bm{p}_{s(t-s)/t}\left( b - \frac st (a+b)\right)
	\quad\text{for all $0<s<t$ and $a,b\in\R$},\footnote{%
	In fact, both sides of \eqref{PPPP}
	represent the probability density of $X_s$ at $b$ where $X$
	denotes a Brownian bridge that emenates from zero and is conditioned to reach  $a+b$ at time $t$.
	}
\end{equation}
equation \eqref{mild} can be recast as the following random evolution equation for $U$:
\begin{align}\label{E:U}
	U(t\,,x) = 1 +\int_{(0,t)\times\R}
	U(s\,,y) \bm{p}_{s(t-s)/t}\left(y- \frac{s}{t}x \right)
	\eta(\d s\, \d y).
\end{align}

The purpose of this paper is to study asymptotic properties of the stationary
process $U(t)$, equivalently $u(t)/\bm{p}_t$.  The main results are stated as
the following three theorems.

\begin{theorem}\label{ergodicity}
		 The process $U(t)$ is weakly mixing, hence also ergodic, for every $t>0$.
\end{theorem}

It follows immediately from \eqref{E:U} that $\E[U(t\,,x)]=1$. Therefore, Theorem \ref{ergodicity}
and the ergodic theorem together imply that for all $t\ge0$,
\begin{equation}\label{erg:thm}
	\lim_{N\to\infty} \frac1N\int_0^NU(t\,,x)\,\d x=1
	\qquad\text{a.s.\ and in $L^1(\Omega)$}.
\end{equation}
In fact, Lemma \ref{lem:||S||} below implies
that \eqref{erg:thm} holds in $L^k(\Omega)$
for every $k\ge1$.

The next two theorems {describe} the rate of convergence in the ergodic theorem \eqref{erg:thm}.
In order to state those theorems, let us introduce
\begin{equation} \label{average}
	\mathcal{S}_{N,t} := \frac 1N \int_0^N [ U(t\,,x) -1] \,\d x
	\qquad\text{for all $N>0$ and $t\ge0$}.
\end{equation}
Then we have the following quantitative central limit theorem.

\begin{theorem}\label{TVD}
	For every $t>0$ there exists
	a real number $c = c(t) >0$ and $N_0=N_0(t)>\e$
	such that for all $N \geq  N_0$,
	\begin{align}\label{TVDeq}
		d_{\rm TV} \left(  \frac{ \mathcal{S}_{N,t}}
		{ \sqrt{{\rm Var}(\mathcal{S}_{N,t})}} ~,~
		{\rm N}(0\,,1)\right) \leq  c\,\sqrt{\frac{\log N}{ N}},
	\end{align}
	where $d_{\rm TV}$ denotes the total variation distance, and ${\rm N}(\mu\,,\sigma^2)$
	denotes the normal law with mean $\mu\in\R$ and variance $\sigma^2>0$.
\end{theorem}

Theorem \ref{TVD} tacitly implies also that $\Var(\mathcal{S}_{N,t})>0$ for all
$N$ large.  As part of the proof of Theorem \ref{TVD}, we in fact prove in
Proposition \ref{pr:Cov:asymp} below that
\begin{equation}\label{AVar}
	\Var(\mathcal{S}_{N,t})\sim\frac{2t\log N}{N}
	\qquad\text{as $N\to\infty$}.
\end{equation}
Therefore, Theorem \ref{TVD}  implies that, for all $t>0$,
\begin{equation}  \label{CLT1}
	\sqrt{\frac{N}{\log N}}\, \mathcal{S}_{N,t}
	\xrightarrow{\text{\rm d}\,}{\rm N} (0\,,2t)
	\qquad\text{as $N\to\infty$}.
\end{equation}
where ``$\xrightarrow{\text{\rm d}\,}$'' denotes convergence in distribution.
Since the limiting variance  $2t$ is a linear function of $t$, the above
suggests the existence of a functional CLT with a Brownian limit. This is
confirmed by the next result of this section.

\begin{theorem}\label{th:FCLT}
	Choose and fix a real number $T>0$.
	Then, as $N\to\infty$,
	\begin{equation}\label{FCLT}
		\sqrt{\frac{N}{\log N}}\,
		\mathcal{S}_{N,\bullet}
		\xrightarrow{C[0,T]} \sqrt 2 B,
	\end{equation}
	where $B$ denotes a standard one-dimensional Brownian motion, and
	``$\xrightarrow{C[0,T]}$'' denotes weak convergence in the Banach space
	$C[0\,,T]$ of all continuous, real-valued functions on $[0\,,T]$, endowed
	with the {uniform} topology.
\end{theorem}

Theorem \ref{TVD} indicates the convergence in total variation distance
of the one-dimensional laws. It seems conceivable that one can obtain the convergence in total variation distance
of the finite-dimensional distributions. 
Moreover, one might wonder if the weak convergence to
Brownian motion in Theorem \ref{th:FCLT} can be replaced by convergence in
total variation. We leave this question as an open problem for the interested
readers. \bigskip

{\noindent\bf Open problem:~} Does the process $\{
	\sqrt{N/\log N}\,\mathcal{S}_{N,t}\}_{t\in[0,T]}$ converge to $\{\sqrt{2}\,
B_t\}_{t\in[0,T]}$ in total variation, as $N\to\infty$, for any $T>0$?
\bigskip

Now let us compare our work with the existing ones to show the difficulties and
hence the contributions of the current paper.  First, regarding Theorem
\ref{ergodicity}, in Chen et al \cite{CKNP}, we used Poincar\'{e}-type
inequalities and Malliavin calculus in order to establish the spatial
ergodicity for a large class of parabolic stochastic PDEs that include the
parabolic Anderson model with  {\em flat initial condition} $u(0) \equiv 1$.
Broadly speaking, the method in \cite{CKNP} is also employed here in order to
prove Theorem \ref{ergodicity}.  However, because the initial profile of
\eqref{PAM} is the singular measure $\delta_0$, novel technical issues arise.
Chief among them is the fact that the Malliavin derivative of the solution to
\eqref{PAM} behaves radically differently from the case with constant initial
data.  This can be seen by comparing our Lemma \ref{derivative:estimate} with
Theorem 6.4 of \cite{CKNP}. As a result, the Poincar\'{e}-type inequality [see
\eqref{Poincare:Cov}] yields a $(\log N/N)$-decay rate, which is bigger than
the $1/N$-rate obtained in the flat case \cite{CKNP}, and the asymptotic
variance \eqref{AVar} is likewise different from the case of flat initial data.
The Poincar\'{e}-type inequality  \eqref{Poincare:Cov} is based on the
Clark-Ocone formula, and the latter plays an import role not only in this
context, but in fact throughout the paper.

Secondly, for Theorem \ref{TVD}, such total variation estimates for spatial
averages of solutions to parabolic stochastic PDEs were introduced by Huang,
Nualart, and Viitasaari \cite{HNV2018} for the one-dimensional stochastic heat
equation driven by a  space-time white noise, and  later extended in Huang,
Nualart,  Viitasaari, and Zheng \cite{HNVZ2019} to the  multidimensional
stochastic heat equation driven by a noise whose spatially homogeneous
covariance is a suitable Riesz kernel.  The main ingredient in deriving such
estimates is the Malliavin-Stein approach (see Nourdin and Peccati
\cite{NP09,NP}) which provides a convergence rate, in total variation distance,
using a combination of Malliavin calculus and Stein's method for normal
approximations.  But unlike the case considered in Huang et al \cite{HNV2018},
where the initial condition was $u(0)\equiv1$, in our setting  the solution to
\eqref{PAM} with delta initial condition is scaled by the heat kernel, and this
produces asymptotic variance for spatial averages of order $\log (N)/N$; see
\eqref{AVar}. As a consequence, we need to normalize the average in
\eqref{CLT1} by the unconventional rate $\sqrt{N/\log N}$.
Moreover, the $\sqrt{\log N/N}$-rate of convergence
of the total variation distance in Theorem \ref{TVD}
is a natural one, which is of the same order as $\sqrt{\Var(\mathcal{S}_{N,t})}$ as $N\to\infty$ (see \eqref{AVar}).
Such a relation also holds in the context of Malliavin-Stein approach to central limit theorems for other types of SPDEs; 
see \cite{DNZ2018, HNV2018, HNVZ2019, NZ2021}.
Furthermore, the presence of these unexpected logarithmic factors is new in the literature, which shows the slow decorrelation of the random field $U(t)$, and can be attributed to the singularity of the delta initial condition. 

Lastly,  the functional central limit theorem stated in Theorem \ref{th:FCLT}
is the counterpart in our framework of  Theorem 1.2 in \cite{HNV2018}. The
convergence in law of finite-dimensional distributions is obtained using the
Malliavin-Stein approach as in the proof of Theorem 1.2 in \cite{HNV2018}, but
the proof of tightness, however,  is more involved due to the singularity of
the initial condition and requires computations which are different from those
in \cite{HNV2018} (see the proof of Proposition \ref{pr:tightness}).  \bigskip

In the following, after introducing some preliminaries in \S\ref{sec:pre}, we
first prove Theorem \ref{ergodicity} in \S\ref{sec:ergodicity}.  Then we
establish an asymptotic results for the covariance of $\mathcal{S}_{N,t}$ in
\S\ref{sec:Asym}, which will be used in the proof of Theorems \ref{TVD} and
\ref{th:FCLT} in \S\ref{Sec:TVD} and \S\ref{sec:6}, respectively. Finally, some
technical lemmas are proved in Appendix.

Let us close the Introduction with a brief description of the notation of this paper.
For every $Z\in L^k(\Omega)$, we write $\|Z\|_k$ instead of
$(\E[|Z|^k])^{1/k}$. Let $\lip$ denote the class of all Lipschitz-continuous,
real-valued  functions on $\R$, and define for all $g: \R\to\R$,
\[
	\lip(g) := \sup_{-\infty<a<b<\infty}
	\frac{|g(b)-g(a)|}{|b-a|}.
\]
Thus, $g\in\lip$ if and only if $\lip(g)<\infty$. Recall that if $g\in\lip$,
then Rademacher's theorem (see Federer \cite[Theorem 3.1.6]{Federer})
ensures that $g$ has a weak derivative
whose essential supremum is $\lip(g)$.
Let $g'$ denote a given measurable version of that derivative.
Throughout, we define
\[
	\log_+(x):=\log(\e+x)\qquad\text{for every $x\ge0$}.
\]
We also use ``$\widehat{\phantom{f}}$'' to denote the Fourier transform,
normalized so that
\[
	\hat{f}(x) = \int_{-\infty}^\infty
	\e^{ixy}f(y)\,\d y\qquad\text{for all $x\in\R$ and $f\in L^1(\R)$.}
\]

\section{Preliminaries}
\label{sec:pre}
\subsection{Clark-Ocone formula}
Let $\mathcal{H}= L^2(\R_+ \times \R)$.
The Gaussian family $\{ W(h)\}_{h \in \mathcal{H}}$ formed by the Wiener integrals
\[
	W(h)= \int_{\R_+\times \R}  h(s\,,x)\, \eta(\d s\, \d x)
\]
defines an  {\it isonormal Gaussian process} on the Hilbert space $\mathcal{H}$.
In this framework we can develop the Malliavin calculus (see Nualart \cite{Nualart}).
We denote by $D$ the derivative operator.
Let $\{\mathcal{F}_s\}_{s \geq 0}$ denote the filtration generated by
the space-time white noise $\eta$.

We recall the following Clark-Ocone formula  (see Chen et al \cite[Proposition 6.3]{CKNP}):
\[
	F= \E [F]  + \int_{\R_+\times\R}
	\E\left[D_{s,y} F \mid \mathcal{F}_s\right] \eta(\d s \, \d z)
	\qquad\text{a.s.},
\]
valid for every random variable $F$ in the Gaussian Sobolev space $\mathbb{D}^{1,2}$.
Thanks to Jensen's inequality for conditional expectations, the above Clark-Ocone formula
readily yields the following Poincar\'e-type inequality, which plays an important
role throughout the paper:
\begin{equation} \label{Poincare:Cov}
	|\Cov(F\,, G)| \le \int_0^{\infty}\d s \int_{-\infty}^\infty\d z\
	\left\| D_{s,z } F  \right\|_2
	\left\| D_{s,z}G \right\|_2
	\qquad\text{for all $F,G\in\mathbb{D}^{1,2}$.}
\end{equation}

\subsection{Malliavin derivative of \texorpdfstring{$u(t\,,x)$}{u(t,x)}}
According to Chen, Hu, and Nualart \cite[Proposition 5.1]{CHN}
(see Chen and Huang \cite[Proposition 3.2]{CH19Densiy} for the higher-dimensional case),
\[
	u(t\,, x)\in\bigcap_{k\ge 2}\mathbb{D}^{1, k}
	\qquad\text{for all   $t>0$        and $x\in\R$},
\]
and the corresponding Malliavin derivative $Du(t\,,x)$ satisfies the
following stochastic integral equation: For $s\in(0\,,t)$,
\begin{align*}
	D_{s, y}u(t\,,x) = \bm{p}_{t - s}(x - y)u(s\,,y) + \int_{(s,t)\times\R}
	\bm{p}_{t - r}(x - z)D_{s, y}u(r\,,z)\,\eta({\d r\, \d z})\qquad\text{a.s.}
\end{align*}
We offer the following estimate on the Malliavin derivative of $u(t\,,x)$.
\begin{lemma}\label{derivative:estimate}
	For every $T >0$ and $k \geq 2$, there exists a real number $C_{T, k} >0$
	such that for $t\in(0\,,T)$ and $x \in \R$, and for almost every $(s\,, y)
	\in (0\,, t) \times \R$,
	\begin{align}\label{Du(t,x)}
		\|D_{s, y}u(t\,,x)\|_k \leq C_{T, k}\,\bm{p}_{t - s}(x - y)\bm{p}_s(y).
	\end{align}
\end{lemma}

\begin{proof}
	The proof is similar to the proof of Theorem 6.4 of Chen et al \cite{CKNP}.
	Fix $t\in(0\,,T)$ and $x \in \R$. Let $u_0(t \,, x) = \bm{p}_t(x)$
	for every $x\in\R$, and define iteratively, for every $n\in\mathbb{Z}_+$,
	\begin{equation}\label{dn1}
		u_{n+1}(t\,,x) := \bm{p}_t(x) +
		\int_{(0,t)\times\R}\bm{p}_{t-r}(x-z)u_n(r\,,z) \,\eta(\d r\,\d z).
	\end{equation}
	Conus, Joseph, Khoshnevisan, and Shiu \cite[Theorem 3.3]{CJKS} and Chen and
	Dalang \cite[Theorem 2.4]{CD15} found independently, and at the same time,
	that there exists a real number $c_{T, k} > 0$ such that for all $(s\,, y)
	\in (0\,, T] \times \R$,
	\begin{equation}\label{||u(s,y)||}
		\sup_{n\in\mathbb{Z}_+}
		\|u_n(s\,, y)\|_k  \vee \|u(s\,, y)\|_k \leq c_{T, k}\, \bm{p}_s(y).
	\end{equation}
	We apply the properties of the divergence operator \cite[Prop.\ 1.3.8]{Nualart}
	in order to deduce from \eqref{dn1} that for almost every $(s\,, y) \in (0\,, t)\times \R$,
	\begin{equation}\label{derivative}
		D_{s, y}u_{n + 1}(t\,,x) = \bm{p}_{t - s}(x - y)u_n(s\,,y) +
		\int_{(s,t)\times\R}\bm{p}_{t - r}(x - z)D_{s, y}u_n(r\,,z)\,\eta(\d r\,\d z)
		\qquad\text{a.s.}
	\end{equation}
	By \eqref{derivative}, \eqref{||u(s,y)||}, and a suitable form of the
	Burkholder-Davis-Gundy inequality (BDG),
	\begin{equation}\label{||Du(t,x)||}
		\|D_{s, y}u_{n+1}(t\,,x)\|_k^2 \leq 2c_{T,k}^2\,
		\bm{p}^2_{t - s}(x - y)\bm{p}^2_s(y) +
		2c_k \int_s^t\d r\int_{-\infty}^\infty\d z\ \bm{p}^2_{t - r}(x - z)\|D_{s, y}u_n(r\,,z)\|_k^2,
	\end{equation}
	where $c_k=4k$; see \cite[(5.6)]{CKNP}.
 	Let $C_k := (2c_{T,k}^2)\vee(2c_k)$. We can iterate \eqref{||Du(t,x)||} to find that
	\begin{align}
		& \|D_{s, y}u_{n + 1}(t\,,x)\|_k^2 \nonumber \\
		& \quad\leq C_k\,\bm{p}^2_{t - s}(x - y)\bm{p}^2_s(y) +
			C_k^2\bm{p}^2_s(y)\int_s^t\d r_1\int_{-\infty}^\infty\d z_1\
			\bm{p}^2_{t - r_1}(x - z_1)\bm{p}^2_{r_1 - s}(z_1 - y)\nonumber \\
		&\quad \quad + \cdots + C_k^{n}\bm{p}^2_s(y)
			\int_s^t\d r_1\int_{-\infty}^\infty\d z_1\int_s^{r_1}\d r_2
			\int_{-\infty}^\infty\d z_2 \cdots\int_s^{r_{n -2}}\d r_{n-1}
			\int_{-\infty}^\infty\d z_{n-1}\ \bm{p}^2_{t - r_1}(x - z_1) \nonumber \\
		& \hskip1in \times \bm{p}^2_{r_1 - r_2}(z_1 - z_2)\times \cdots\times
			\bm{p}^2_{r_{n -1} - s}(z_{n -1} - y)\nonumber  \\
		&\quad \quad +  C_k^{n+1}\bm{p}^2_s(y)\int_s^t\d r_1
			\int_{-\infty}^\infty\d z_1\int_s^{r_1}\d r_2
			\int_{-\infty}^\infty\d z_2 \cdots\int_s^{r_{n -1}}\d r_n
			\int_{-\infty}^\infty\d z_n\  \bm{p}^2_{t - r_1}(x - z_1)\nonumber \\
		& \hskip1in \times \bm{p}^2_{r_1 - r_2}(z_1 - z_2)\times \cdots\times
			\bm{p}^2_{r_{n -1} - r_n}(z_{n -1} - z_n)\bm{p}^2_{r_{n} - s}(z_{n} - y). \label{Du_{n+1}}
	\end{align}
	In order to simplify the preceding expression, let us first use the elementary identity \eqref{PPPP}
	in order to see that
	\[
		\int_{-\infty}^\infty\bm{p}^2_{t -s}(x-y)\bm{p}^2_{s -r}(y-z)\,\d y
		= \sqrt{\frac{t -r}{4\pi (t -s)(s -r)}}\,\bm{p}^2_{t -r}(x-z).
	\]
	Consequently,
	 \begin{align}
		&\int_s^t\d r_1\int_{-\infty}^\infty\d z_1\int_s^{r_1}\d r_2
			\int_{-\infty}^\infty\d z_2 \cdots\int_s^{r_{n -1}}\d r_n
			\int_{-\infty}^\infty\d z_n\nonumber \\
		& \hskip.5in  \bm{p}^2_{t - r_1}(x - z_1) \bm{p}^2_{r_1 - r_2}(z_1 - z_2)\times \cdots\times
			\bm{p}^2_{r_{n -1} - r_n}(z_{n -1} - z_n)\bm{p}^2_{r_{n} - s}(z_{n} - y) \nonumber \\
		&\quad = (4\pi)^{-n/2}\,\bm{p}^2_{t-r}(x-y)\int_s^t\d r_1
			\int_s^{r_1}\d r_2\cdots\int_s^{r_{n-1}}\d r_n\,
			\sqrt{\frac{t -s}{(t -r_1)(r_1 -r_2)\cdots(r_{n-1}- r_n)(r_n -s)}}  \nonumber \\
		& \quad = \left(\frac{t-s}{4\pi}\right)^{n/2}
			\bm{p}^2_{t-r}(x-y)\int_{0<r_n<\cdots<r_1<1}
			\frac{\d r_1\cdots\d r_n}{\sqrt{(1-r_1)(r_1-r_2)\cdots r_n}} \nonumber \\
		& \quad = \left(\frac{t-s}{4\pi}\right)^{n/2}
			\frac{\Gamma(1/2)^{n}}{\Gamma(n/2)}\,
			\bm{p}^2_{t - s}(x - y). \label{Du_{n+1}'}
	\end{align}
	Together, \eqref{Du_{n+1}} and \eqref{Du_{n+1}'} yield
	\begin{align*}
		\|D_{s, y}u_{n + 1}(t\,,x)\|_k^2 &\leq
			\bm{p}^2_{t - s}(x - y)\bm{p}^2_s(y)  \sum_{j = 0}^n C_k^{j+ 1}
			\left(\frac{t-s}{4\pi}\right)^{j/2}
			\frac{\Gamma(1/2)^{j}}{\Gamma(j/2)} \\
		& \leq \bm{p}^2_{t - s}(x - y)\bm{p}^2_s(y)  \sum_{j = 0}^\infty
		\frac{C_k^{j+ 1}T^{j/2}}{(4\pi)^{j/2}}\frac{\Gamma(1/2)^{j}}{\Gamma(j/2)}.
	\end{align*}
	Since the above series is convergent, we can conclude that there exists $c'_{T, k} > 0$
	such that for almost every $(s\,, y)\in (0\,, t) \times \R$,
	\begin{align}\label{Du_n}
		\sup_{n \geq 0}\|D_{s, y}u_{n}(t\,,x)\|_k &\leq c'_{T, k}\, \bm{p}_{t - s}(x - y)\bm{p}_s(y).
	\end{align}
	Moreover, \eqref{PPPP} and \eqref{Du_n} together yield
	\begin{equation}\label{||Du_n||}\begin{split}
		\sup_{n \geq 0}\E\left(\|Du_{n}(t\,,x)\|_{\mathcal{H}}^2\right)
			&\le (c_{T,2}')^2\int_0^t\d s\int_{-\infty}^\infty\d y\ \bm{p}^2_{t - s}(x - y)\bm{p}^2_s(y)\\
		&=(c_{T,k}')^2\bm{p}^2_t(x)\int_0^t\d s\int_{-\infty}^\infty\d y\
			\bm{p}_{s(t-s)/t}^2\left( y - \frac st x\right)\\
		&=(c_{T,k}')^2\bm{p}^2_t(x)\int_0^t \sqrt{\frac{t}{4\pi s(t-s)}}\,\d s<\infty,
	\end{split}\end{equation}
	where we have used the semigroup property of the heat kernel in the final
	identity.  It follows from  \eqref{||Du_n||} and the closability properties
	of the Malliavin derivative that  there exists a subsequence $n(1) <
	n(2)<\cdots$ of positive integers such that $Du_{n(\ell)}(t\,,x)$  converges
	to $Du(t\,,x)$ in the weak topology of $L^2(\Omega\, ; \mathcal{H})$.  Then,
	we use a smooth approximation $\{\psi_\varepsilon\}_{\varepsilon>0}$ to the
	identity in $\R_+\times \R$,  and apply Fatou's lemma and duality for
	$L^k$-spaces, in order to find that for almost every $(s\,,y) \in (0\,,t)
	\times \R$ and for all $k\ge 2$,
	\begin{align*}
		\|D_{s,y}u(t\,,x) \|_k & \le \liminf_{\varepsilon \to 0}
			\left \| \int_0^\infty\d s'\int_{-\infty}^\infty\d y'\, D_{s',y'} u(t\,,x)
			\psi_\varepsilon(s-s', y-y')\right\|_k\\
		& \le  \liminf_{\varepsilon \to 0}
			\sup_{\|G \|_{k/(k - 1)}\le 1}
			\left| \int_0^\infty\d s'\int_{-\infty}^\infty\d y'\,  \E\left[ G D_{s',y'} u(t\,,x) \right]
			\psi_\varepsilon(s-s', y-y') \right|.
	\end{align*}
	Choose and fix a random variable $G\in L^{2}(\Omega)$ such that $\| G
	\|_{k/(k-1)} \le 1$.  Because $Du_{n(\ell)}(t\,,x)$ converges weakly in
	$L^2(\Omega\,;\HH)$ to $Du(t\,,x)$ as $\ell\to\infty$, we can write
	\begin{align*}
		& \left| \int_0^\infty\d s'\int_{-\infty}^\infty\d y'\, \E \left[ G D_{s',y'} u(t\,,x) \right]
			\psi_\varepsilon(s-s', y-y')  \right|  \\
		&\hskip1.5in= \lim _{\ell\rightarrow \infty} \left|
			\int_0^\infty\d s'\int_{-\infty}^\infty\d y'\,  \E\left[ G D_{s',y'} u_{n(\ell)}(t\,,x) \right]
			\psi_\varepsilon(s-s', y-y') \right| \\
		&\hskip1.5in\le\limsup_{\ell\to\infty}
			\int_0^\infty\d s'\int_{-\infty}^\infty\d y'\, \left\| D_{s',y'} u_{n(\ell)}(t\,,x) \right\|_k
			\psi_\varepsilon(s-s', y-y')  \\
		&\hskip1.5in\le c'_{T, k}
			\int_0^\infty\d s'\int_{-\infty}^\infty\d y'\,  \bm{1}_{(0,t)}(s')
			\bm{p}_{t - s'}(x - y')\bm{p}_{s'}(y')
			\psi_\varepsilon(s-s', y-y').
	\end{align*}
	Let $\varepsilon\to 0$ to conclude the proof of \eqref{Du(t,x)}.
\end{proof}

\subsection{The Malliavin-Stein method}

Recall that if $X$ and $Y$ are random variables with respective probability distributions $\mu$ and $\nu$ on
$\R$, then the total variation distance between $X$ and $Y$ is defined as
\[
	d_{\rm TV}  (X\,,Y)= \sup_{B\in \mathcal{B}(\R)} | \mu(B)- \nu(B)|,
\]
where  $\mathcal{B}(\R)$ denotes the family of all Borel subsets of $\R$. The
same sort of definition continues to hold when $X$ and $Y$ are abstract random variables
on a topological space $\mathbb{X}$, except $\mathcal{B}(\R)$ is replaced
by $\mathcal{B}(\mathbb{X})$.

We abuse  notation and let $d_{\rm TV}(F\,,{\rm N}(0\,,1))$  denote the total
variation distance between the law of $F$ and the ${\rm N}(0\,,1)$ law.  The
following bound on $d_{\rm TV}(F\,,{\rm N}(0\,,1))$ follows from a suitable
combination of ideas from the Malliavin calculus and Stein's method for normal
approximations; see Nualart and Nualart \cite[Theorem 8.2.1]{NN}.

\begin{proposition}\label{pr:MS}
	Suppose that $F\in \mathbb{D}^{1,2}$ satisfies $\E(F^2)=1$ and  $F=\delta(v)$
	for some $v$ in the $L^2(\Omega)$-domain  of the divergence operator $\delta$.
	Then,
	\[
		d_{\rm TV} (F\,,  {\rm N}(0\,,1)) \le
		2\sqrt{ {\rm Var} \left ( \langle DF\,, v \rangle_{\mathcal{H}} \right) }.
	\]
\end{proposition}

In the proof of Theorem \ref{FCLT} we will make use of the following
generalization of  a result of Nourdin and Peccati \cite[Theorem 6.1.2]{NP}.

\begin{proposition}\label{lemma: NP 6.1.2}
	Let $F=( F^{(1)}, \dots, F^{(m)})$ be a random vector such that, for every $i=1,\ldots,m$,
	$F^{(i)} = \delta (v^{(i)})$ for some $v^{(i)} \in {\rm Dom}\, [\delta]$.
	Assume additionally that
	$F^{(i)} \in \mathbb{D}^{1,2}$ for $i=1,\dots,m$. Let $G$ be a centered
	$m$-dimensional Gaussian random vector with covariance matrix $(C_{i,j}) _{1\le i,j \le m} $.
	Then, for every $h\in C^2(\R^m)$ that has bounded second partial derivatives,
	\[
		| \E( h(F)) -\E (h(G)) | \le \tfrac 12 \|h ''\|_\infty
		\sqrt{   \sum_{i,j=1}^m   \E \left( \left|
		C_{i,j} - \langle DF^{(i)}\,, v^{(j)} \rangle_{\HH} \right|^2
		\right)},
	\]
	where
	\[
		\|h'' \| _\infty := \adjustlimits\max_{1\le i,j \le m}
		\sup_{x\in\R^m}  \left| \frac { \partial ^2h (x) } {\partial x_i \partial x_j} \right|.
	\]
\end{proposition}

\subsection{On the ergodic theorem (\ref{erg:thm})}

Recall the definition \eqref{average} of $\mathcal{S}_{N,t}$ and observe that
the ergodic theorem (\ref{erg:thm}) can be recast in terms of the average
integral $\mathcal{S}_{N,t}$ as follows:
\[
	\lim_{N\to\infty}\mathcal{S}_{N,t}=0\quad\text{a.s.\ and in $L^1(\Omega)$}.
\]
The following lemma proves that the ergodic theorem \eqref{erg:thm} holds in
$L^k(\Omega)$ for every $k\ge2$, hence also in $L^k(\Omega)$ for every $k\ge1$.
It also yields a quantitative upper bound of $O(\sqrt{\log(N)/N})$ on the rate
of convergence in $L^k(\Omega)$ for every $k\ge1$, with a constant that
describes also the behavior  of the limit uniformly in $t$ when $t\ll1$.
Perhaps not surprisingly, the mentioned rate of convergence coincides with the
rate of convergence to normality that was ensured by Theorem \ref{TVD}.

\begin{lemma}\label{lem:||S||}
	For all real numbers $k\ge2$ and $T>0$ there exists a number $A_{k,T}>0$
	such that
	\[
		\sup_{N\ge\e}\left\| \sqrt{\frac{N}{\log N}}\,\mathcal{S}_{N,t} \right\|_k \le
		A_{k,T}\sqrt{t\log_+(1/t)}
		\quad\text{uniformly for all $t\in(0\,,T)$},
	\]
	where $\log_+(w):=\log(\e+w)$ for all $w\ge0$.
\end{lemma}

\begin{proof}
	Choose and fix a real number $k\ge2$.
	By the BDG inequality and \eqref{PPPP},
	\begin{align*}
		\|\mathcal{S}_{N,t}\|_k^2 
			&=\frac{1}{N^2}\left\| \int_{(0,t)\times\R}
			U(s\,,y) \left[\int_0^N\bm{p}_{s(t-s)/t}\left(y- \frac{s}{t}x \right)\d x\right]
			\eta(\d s\, \d y)\right\|_k^2\\
		&\le\frac{   c_k}{N^2}\int_0^t\d s\int_{-\infty}^\infty\d y\
			\|U(s\,,y)\|_k^2\left[\int_0^N\bm{p}_{s(t-s)/t}\left(y- \frac{s}{t}x \right)\d x\right]^2,
	\end{align*}
	uniformly for all $N,t>0$. Apply \eqref{||u(s,y)||} to see that
	\[
		\|\mathcal{S}_{N,t}\|_k^2
		\le\frac{  c_k\: c_{T,k}^2}{N^2}\int_0^t\d s\int_{-\infty}^\infty\d y\
		\left[\int_0^N\bm{p}_{s(t-s)/t}\left(y- \frac{s}{t}x \right)\d x\right]^2,
	\]
	uniformly for all $N>0$ and $t\in(0\,,T)$. Now expand the square and appeal to the
	semigroup property of the heat kernel in order to find that, for every $N,t>0$,
	\begin{align*}
		\int_{-\infty}^\infty\d y\
			\left[\int_0^N\bm{p}_{s(t-s)/t}\left(y- \frac{s}{t}x \right)\d x\right]^2
    & =\int_0^N\d y\int_0^N\d z\ \bm{p}_{2s(t-s)/t}\left( \frac st(y-z)\right)              \\
    & = \left( \frac ts\right)^2\int_0^{Ns/t}\d a\int_0^{Ns/t}\d b\ \bm{p}_{2s(t-s)/t}(a-b) \\
    & = \frac{Nt}{\pi s} \int_{-\infty}^\infty\left(\frac{1-\cos z}{z^2}\right) \exp\left( -\frac{t(t-s)z^2}{N^2s} \right)\d z;
	\end{align*}
	see Lemma \ref{lem7.2} of the Appendix. Consequently, if
	$N>0$ and  $t\in(0\,,T)$, then
	\begin{align*}
		\|\mathcal{S}_{N,t}\|_k^2 &\le\frac{t c_kc_{k,T}^2}{\pi N}\int_{-\infty}^\infty\d x\:
			\left(\frac{1-\cos x}{x^2}\right)
			\int_0^t\frac{\d s}{s}
			\exp\left( -\frac{t(t-s)x^2}{N^2s} \right)\\
		&=\frac{ c_kc_{k,T}^2\log N}{\pi N}\int_{-\infty}^\infty
			\left(\frac{1-\cos x}{x^2}\right)G_{N,t}(x)\,\d x,
	\end{align*}
	where $G_{N,t}$ is defined in \eqref{G_{N,t}} below, in the Appendix. We may
	appeal to Lemma \ref{G_N:limit} of the Appendix to conclude the result.
\end{proof}

\section{Proof of Theorem \ref{ergodicity}}\label{sec:ergodicity}
Since weak mixing implies ergodicity, it suffices to prove that $U(t)$ is
weakly mixing for every $t>0$. We follow the proof of \cite[Corollary
9.1]{CKNP} in order to reduce the proof of Theorem \ref{ergodicity} to the
verification of the following:
\begin{equation}\label{enough}
	\lim_{|x|\to\infty}\Cov\left[ \mathcal{G}(x)\,,\mathcal{G}(0)\right]=0,
\end{equation}
where  the functions $g_1,\ldots,g_k\in C^1_b(\R)$ satisfy
$g_j(0)=0$ and $\text{\rm Lip}(g_j)=1$ for every $j=1,\ldots,k$,
\[
	\mathcal{G}(x) :=
	\prod_{j=1}^k g_j ( U(t\,,x+\zeta^j))\qquad
	\text{for all $x\in\R$},
\]
and $\zeta^1,\ldots,\zeta^k$ are fixed real numbers.
Thus, it suffices to prove \eqref{enough}.

By the  chain rule
for the Malliavin derivative \cite[Proposition 1.2.4]{Nualart},
\[
	D_{s,z} \mathcal{G}(x) = \sum_{j_0=1}^k
	\left(\prod_{\substack{j=1\\j \not= j_0}}^{k}
	g_j \left( U(t\,,x+\zeta^j) \right) \right)
	g_{j_0}' \left(U(t\,,x+\zeta^{j_0})\right)
	D_{s,z} U(t\,, x+ \zeta^{j_0}).
\]
Therefore, the definition of the process $U$ in \eqref{U}, \eqref{||u(s,y)||},
and Lemma \ref{derivative:estimate} together imply the existence of a real number $c=c(T,k)$ such that
\begin{align*}
	\| D_{s,z} \mathcal{G}(x) \|_2
		&\le  \sum_{j_0=1}^k  \left(\prod_{j=1, j \not= j_0}^{k}
		\|  g_j ( U(t\,,x+\zeta^j))  \|_{2k} \right)  \| D_{s,z} U(t\,, x+ \zeta^{j_0})\|_{2k} \\
	&\le c\sum_{j=1}^k  \frac{\bm{p}_{t-s} (x+ \zeta^{j} -z)
		\bm{p}_{s}(z)}{\bm{p}_t(x + \zeta^j)} = c\sum_{j=1}^k
		\bm{p}_{s(t - s)/t}\left(z - \frac{s}{t}(x + \zeta^j)\right),
\end{align*}
uniformly for all $0<s<t\le T$ and $x,z\in\R$;
the equality holds due to \eqref{PPPP}.  Now apply the Poincar\'e inequality \eqref{Poincare:Cov}
and the semigroup property of the heat kernel to see that
\[
	\left| {\rm Cov} \left[  \mathcal{G}(x)\,, \mathcal{G}(0)\right] \right|
	\le c^2\sum_{j,\ell = 1}^{k}\int_0^t
	\bm{p}_{2s(t - s)/t}\left(\frac{s}{t}(x  + \zeta^j - \zeta^\ell)\right)
	\d s.
\]
This implies \eqref{enough}, thanks to the
dominated convergence theorem, and concludes the proof.\qed

\section{Asymptotic behavior of the covariance}
\label{sec:Asym}

Recall from \eqref{average} that
\[
	\mathcal{S}_{N,t}= \frac 1N \int_0^N [ U(t\,,x) -1] \,\d x,
\]
where $U(t\,,x)$ was defined in \eqref{U}.  The following proposition provides
the asymptotic behavior of the covariance function of the renormalized sequence
of processes  $\mathcal{S}_{N,t}$ as $N$ tends to infinity.
\begin{proposition}\label{pr:Cov:asymp}
	For every $t_1,t_2>0$,
	\[
		\lim_{N\to\infty}
		\Cov\left[  \sqrt{\frac{N}{\log N}}\, \mathcal{S}_{N,t_1} ~,~
		\sqrt{\frac{N}{\log N}}\, \mathcal{S}_{N,t_2}\right] =
		2(t_1 \wedge t_2).
	\]
\end{proposition}

\begin{proof}
	First, let us recall from Chen and Dalang \cite[(2.31)]{CD15}
	that, for all $s >0$ and $z \in \R$,
	\begin{align}\label{squareformula}
		\E\left( |u(s\,, z)|^2\right) = \bm{p}^2_s(z)(1 + \theta(s)),
	\end{align}
	where
	\begin{align}\label{theta}
		\theta(s) := \e^{s/4}\sqrt{s/2}
		\int_{-\infty}^{\sqrt{s/2}}\e^{- y^2/2}\,\d y
		\qquad\text{for all $s>0$}.
	\end{align}
	By \eqref{mild}, the It\^{o}-Walsh isometry, and \eqref{squareformula},
	\begin{align*}
		\Cov\left[ U(t_1\,,x) ~,~ U(t_2\,,y)\right] &=
			\frac{1}{\bm{p}_{t_1}(x)\bm{p}_{t_2} (y)}
			\int_0^{t_1\wedge t_2} \d s\int_{-\infty}^\infty \d z\
			\bm{p}_{t_1 - s}(x - z)\bm{p}_{t_2 - s}(y - z)\E\left( |u(s\,, z)|^2\right) \\
		& =\frac{1}{\bm{p}_{t_1}(x)\bm{p}_{t_2}(y)}
			\int_0^{t_1\wedge t_2}\d s
			\int_{-\infty}^\infty \d z\
			\bm{p}_{t_1 - s}(x - z)\bm{p}_{t_2 - s}(y - z)\bm{p}^2_s(z)(1 + \theta(s)) \\
		& = \int_0^{t_1\wedge t_2} \d s
			\int_{-\infty}^\infty \d z\
			\bm{p}_{s(t_1 - s)/t_1}\left(z - \frac{s}{t_1}x\right)
			\bm{p}_{s(t_2 - s)/t_2}\left(z - \frac{s}{t_2}y\right)(1 + \theta(s)) \\
		& = \int_0^{t_1\wedge t_2}
			\bm{p}_{s[(t_1 - s)/t_1+(t_2-s)/t_2]}
			\left(s\left[\frac{x}{t_1}- \frac y{t_2}\right]\right)(1 + \theta(s))\,\d s \\
		& =: \int_0^{t_1\wedge t_2} \mathcal{P}_{s,t_1,t_2}(x,y)(1 + \theta(s))\,\d s,
	\end{align*}
	notation being clear from context.
	Let $\tau := 2t_1t_2/(t_1+t_2)$, so that we can write
	\[
		\mathcal{P}_{s,t_1, t_2}(x\,,y)= \mathcal{P}_{s,\tau}\left( \frac {2(xt_2-yt_1)}{t_1+t_2} \right)
		\quad\text{for}\quad
		\mathcal{P}_{s, t}(w)= \bm{p}_{2s(t-s)/t}\left(\frac{sw}{t}\right).
	\]
	If $t_1 < t_2$, then
	\begin{align*}
		&\Cov\left[\sqrt{\frac{N}{\log N}}\, \mathcal{S}_{N,t_1}
			~,~\sqrt{\frac{N}{\log N}}\,\mathcal{S}_{N,t_2} \right]
			= \frac{1}{ N\log N}\int_0^N\d y\int_0^N \d x\
			\Cov\left[ U(t_1\,,x) \,, U(t_2\,,y)\right]  \\
		&\hskip1in = \frac{1 }{ N\log N}\int_0^{t_1}\d s \, (1 + \theta(s))
			\int_0^{N}\d y \int_0^{N}\d x\
			\mathcal{P}_{s,\tau}\left( \frac {2(xt_2-yt_1)}{t_1+t_2} \right).
	\end{align*}
 	In order to simplify the exposition define
	\[
		\tau_1:=\frac {2t_2}{ t_1+t_2}
		\quad\text{and}\quad
		\tau_2=\frac {2t_1}{ t_1+t_2}.
	\]
	We then change  variables [$x \to x/ \tau_1$ and $y\to y/\tau_2$] to obtain
	\begin{align*}
		&\Cov\left[\sqrt{\frac{N}{\log N}}\, \mathcal{S}_{N,t_1}
			~,~ \sqrt{\frac{N}{\log N}}\,\mathcal{S}_{N,t_2} \right]  \\
		&\quad = \frac{1 }{\tau_1 \tau_2 N \log N}\int_0^{t_1}(1 + \theta(s))\,\d s
			\int_0^{N\tau_1}\d x \int_0^{N\tau_2}\d y\    \mathcal{P}_{s,\tau}(x-y)\\
		&\quad = \frac{ \tau }{\tau_1 \tau_2 N \log N}
			\int_0^{t_1}\left(\frac{1+\theta(s)}{s}\right)\d s
			\int_0^{N\tau_1}\d x \int_0^{N\tau_2}\d y\
			\bm{p}_{2\tau( \tau-s)/s} (x-y),
	\end{align*}
	where in the last equality we have used the scaling property,
	\begin{equation}\label{p-scale}
		\bm{p}_\sigma(\alpha w) = \alpha^{-1} \bm{p}_{\sigma/\alpha^2}(w),\qquad\text{%
		valid for all
		$\sigma,\alpha>0$ and $w\in\R$}.
	\end{equation}
	Since $\widehat{ \mathbf{1}_{[0,a]} } (\xi) =a \widehat{ \mathbf{1}_{[0,1]}
	}(a\xi)$ for all $a>0$ and $\xi\in\R$, Parseval's identity ensures that
	\begin{align*}
		\Cov&\left[\sqrt{\frac{N}{\log N}}\, \mathcal{S}_{N,t_1}
			~,~\sqrt{\frac{N}{\log N}}\,\mathcal{S}_{N,t_2} \right]  \\
				& = \frac{\tau}{2\pi  \tau_1 \tau_2 \log N}\int_0^{t_1}
			\left(\frac{1+\theta(s)}{s}\right)\d s
			\int_{-\infty}^\infty\widehat{\mathbf{1}_{[0,\tau_1]}}(w)
			\overline{\widehat{\mathbf{1}_{[0,\tau_2]}}(w)}
			\exp\left( -\frac{(\tau - s)\tau}{s}\frac{w^2}{N^2}\right)\d w  \\
		& = \frac{1}{2\pi \tau_1\tau_2}\int_{-\infty}^\infty
			\widehat{\mathbf{1}_{[0,\tau_1]}}(w)
			\overline{\widehat{\mathbf{1}_{[0,\tau_2]}}(w)}
			G_{N, \tau}(w)\, \d w  \\ \notag
		& \qquad - \frac{\tau}{2\pi  \tau_1 \tau_2 \log N}
			\int_{t_1}^{\tau}\frac{\d s}{s}\int_{-\infty}^\infty
			\widehat{\mathbf{1}_{[0,\tau_1]}}(w)
			\overline{\widehat{\mathbf{1}_{[0,\tau_2]}}(w)}
			\exp\left(-\frac{(\tau - s)\tau}{s}\frac{w^2}{N^2}\right)\d w  \\
		& \qquad + \frac{\tau}{2\pi \tau_1\tau_2\log N}
			\int_0^{t_1}\frac{\theta(s)}{s}\,\d s\int_{-\infty}^\infty
			\widehat{\mathbf{1}_{[0,\tau_1]}}(w)
			{\widehat{\mathbf{1}_{[0,\tau_2]}}(w)}
			\exp\left(-\frac{(\tau - s)\tau}{s}\frac{w^2}{N^2}\right)\d w   \\
		& =: A_{N}^{(1)} - A_{N}^{(2)} +A_{N}^{(3)},
	\end{align*}
	where  the function $G_{N,\tau}$  is defined in \eqref{G_{N,t}} below, in the Appendix.
	We plan to prove that
	\begin{equation}\label{AAA}
		\lim_{N\to\infty}A^{(1)}_N= 2t_1
		\quad\text{and}\quad
		\lim_{N\to\infty}A^{(2)}_N=\lim_{N\to\infty}A^{(3)}_N=0.
	\end{equation}
	These facts together conclude the proof of the proposition.

	In order to understand the behavior of $A^{(1)}_N$ we first apply Lemma
	\ref{G_N:limit} and the dominated convergence theorem, and then the Parseval
	identity, in order to verify the first of the three assertions in
	\eqref{AAA}:
	\[
		\lim_{N\to\infty}A_{N}^{(1)}
		= \frac{2\tau}{2\pi\tau_1\tau_2}\int_{-\infty}^\infty
		\widehat{\mathbf{1}_{[0,\tau_1]}}(w)
		\overline{\widehat{\mathbf{1}_{[0,\tau_2]}}(w)}\, \d w
		= \frac{2\tau} {\tau_1\tau_2}
		\left\< \mathbf{1}_{[0,\tau_1]} ~,~ \mathbf{1}_{[0,\tau_2]} \right\>_{L^2(\R)} =2t_1.
	\]

	We study $A_N^{(2)}$ by making a change of variables $[s\to \tau/(s+1)]$ to find that
	\begin{align*}
		A_N^{(2)} &=  \frac{\tau}{2\pi  \tau_1 \tau_1\log N}
			\int_{0}^{(t_2-t_1)/(t_2+t_1)}\frac{\d s}{1+s}
			\int_{-\infty}^\infty\widehat{\mathbf{1}_{[0,\tau_1]}}(w)
			\overline{\widehat{\mathbf{1}_{[0,\tau_2]}}(w)}
			\exp\left( - \frac{\tau sw^2}{N^2}\right)\d w.
	\end{align*}
	Since $\exp(-\tau s w^2/N^2)\le 1$, this proves that $A_N^{(2)} = O(1/\log
	N)\to0$ as $N\to\infty$. Therefore, it remains to prove the third assertion
	in \eqref{AAA} about $A^{(3)}_N$.  For that, we change  variables $[s\to \tau
	s]$ to obtain
	\begin{align*}
		\left|A_{N}^{(3)}\right|&\leq \frac{\tau}{2\pi \tau_1 \tau_2 \log N}
			\int_{-\infty}^\infty
			\left| \widehat{\mathbf{1}_{[0,\tau_1]}}(w)
			\overline{\widehat{\mathbf{1}_{[0,\tau_2]}}(w)}\right|\d w
			\int_0^1\frac{\theta(\tau s)}{s}
			\exp\left( -\frac{(1 - s)\tau}{s}\frac{w^2}{N^2}\right)\d s\\
		& = \frac{\tau}{2\pi \tau_1 \tau_2 \log N}
			\int_{-\infty}^\infty\left|\widehat{\mathbf{1}_{[0,\tau_1]}}(w)
			\overline{\widehat{\mathbf{1}_{[0,\tau_2]}}(w)}\right|  \d w
			\int_0^{\infty}\frac{\theta(\tau/(r + 1))}{r + 1}
			\exp\left( -\frac{w^2\tau r}{N^2} \right) \d r.
	\end{align*}
	By the definition of the function $\theta$ in \eqref{theta},
	\begin{align*}
		\theta\left( \frac{\tau}{r + 1}\right)
		\exp\left( -\frac{w^2\tau r}{N^2} \right)
		<\theta\left( \frac{\tau}{r + 1}\right)\le \e^{\tau/4}\sqrt{\frac{\tau\pi}{r + 1}}
		\quad\text{for all $r>0$}.
	\end{align*}
	Hence,
	\[
		\left|A_{N}^{(3)}\right| \leq
		\frac{\e^{\tau/4}t\sqrt{\tau\pi}}{2\pi \tau_1 \tau_2 \log N}
		\int_{-\infty}^\infty  \left|\widehat{\mathbf{1}_{[0,\tau_1]}}(w)
		\overline{\widehat{\mathbf{1}_{[0,\tau_2]}}(w)}\right|\d w \times
		\int_0^{\infty}\frac{\d r}{(r+1)^{3/2}}\to 0,
	\]
	as $N\to\infty$. This concludes the proof of \eqref{AAA} and hence the  proof of the proposition.
\end{proof}

\section{Proof of Theorem \ref{TVD}}\label{Sec:TVD}

For all $N,t,s>0$ and $y\in\R$ define
\begin{equation}\label{g+v}
	g_{N,t}(s\,,y) := \bm{1}_{(0,t)}(s)\frac1N\int_0^N\bm{p}_{s(t-s)/t}\left( y - \frac st x\right)\d x
	\quad\text{and}\quad
	v_{N,t}(s\,,y) := g_{N,t}(s\,,y)U(s\,,y).
\end{equation}
Because of \eqref{average} and a stochastic Fubini argument,
\begin{equation}\label{S=delta(v)}
	\mathcal{S}_{N,t} = \int_{\R_+\times\R} v_{N,t}(s\,,y)\,\eta(\d s\,\d y)
	=\delta(v_{N,t})\qquad\text{a.s.,}
\end{equation}
owing to the fact that $v_{N,t}$ is an adapted random field and hence its
stochastic integral agrees with its divergence  (see Nualart \cite[Chapter
1.3.3]{Nualart}).  Our work so far shows that $\mathcal{S}_{N,t}$ is Malliavin
differentiable, and that the following defines a version of the Malliavin
derivative of $\mathcal{S}_{N,t}$:
\begin{equation}\label{DS}
	D_{r,z}\mathcal{S}_{N,t} = \bm{1}_{(0,t)}(r) v_{N,t}(r\,,z) +
	\bm{1}_{(0,t)}(r) \int_{(r,t)\times\R} D_{r,z}v_{N,t}(s\,,y)\,\eta(\d s\,\d y).
\end{equation}
The key technical result of this section is the following proposition:

\begin{proposition}\label{Var<Ds,v>}
	For every $T>0$ there exists a real number $K_T>0$ such that
	\[
		\sup_{t,\tau\in(0,T)}
		\Var \left\< D\mathcal{S}_{N,t}\,,v_{N,\tau}\right>_{\HH}
		\le K_T\,\frac{(\log N)^3}{N^3}\qquad\text{for all $N\ge\e$}.
	\]
\end{proposition}

We plan to first prove Proposition \ref{Var<Ds,v>}. Then, we will use this
proposition to prove Theorem \ref{TVD}. The key to the proof of Proposition
\ref{Var<Ds,v>} is the following simple decomposition, which is an immediate
consequence of \eqref{DS}:
\begin{equation}\label{DS:v:X:Y}
	\left\< D\mathcal{S}_{N,t}\,, v_{N,\tau}\right\>_{\HH} =
	\mathcal{X}_{N,t,\tau} +\mathcal{Y}_{N,t,\tau},
\end{equation}
where
\begin{equation}\label{XY}\begin{split}
	\mathcal{X}_{N,t,\tau} &:= \left\< v_{N,t}\,,v_{N,\tau}\right\>_{\HH},\quad\text{and}\\
	\mathcal{Y} _{N,t,\tau} &:= \int_0^\infty\d r\int_{-\infty}^\infty\d z\
		v_{N,\tau}(r\,,z)\left( \int_{(r,t)\times\R} D_{r,z} v_{N,t}(s\,,y)\,\eta(\d s\,\d y)\right).
\end{split}\end{equation}
The decomposition \eqref{DS:v:X:Y} ensures that
\begin{equation}\label{3V}
	\Var\<D\mathcal{S}_{N,t}\,,v_{N,\tau}\>_{\HH} \le 2\Var(\mathcal{X}_{N,t,\tau})
	+ 2\Var(\mathcal{Y}_{N,t,\tau}).
\end{equation}
Therefore, the bulk of the work is to establish bounds on the last two
variances. Those require some effort and are carried out separately, using
slightly different ideas, in Lemmas \ref{lem:Var(X)} and \ref{lem:Var(Y)}
respectively. In light of those lemmas and \eqref{3V}, the proof Proposition
\ref{Var<Ds,v>} is immediate, with no  need for additional proof.

First let us observe that the mean of $\<D\mathcal{S}_{N,t}\,,v_{N,\tau}\>_\HH$
is carried by $\mathcal{X}_{N,t,\tau}$.

\begin{lemma}\label{lem:E<DS,v>}
	For every $T,N>0$ and $t,\tau\in(0\,,T)$,
	\[
		\E\mathcal{Y}_{N,t,\tau}=0\quad\text{and}\quad
		\E \left\< D\mathcal{S}_{N,t}\,, v_{N,\tau}\right\>_{\HH}
		= \E \mathcal{X}_{N,t,\tau} = \Cov\left( \mathcal{S}_{N,t}~,~\mathcal{S}_{N,\tau}\right).
	\]
\end{lemma}

\begin{proof}
	Thanks to Gaussian integration by parts (see Nualart \cite[(1.42)]{Nualart}),
	$\E(\< DF\,,V\>_{\HH}) = \E[F\delta(V)]$ for all $F\in\mathbb{D}^{1,2}$ and
	$V\in\text{\rm Dom}[\delta]$.  Choose $F\equiv1$ to observe the well-known
	fact that $\delta(V)$ has mean zero, and choose $F=\delta(U)$ to see that
	$\E(\<D\delta(U)\,,V\>_{\HH})=\Cov(\delta(U)\,,\delta(V))$ whenever
	$U,V\in\text{\rm Dom}[\delta]$.  Thanks to \eqref{S=delta(v)} we can apply
	the preceding with $U=v_{N,t}$ and $V=v_{N,\tau}$ to see that
	$\mathcal{S}_{N,t}=\delta(U)$ and $\mathcal{S}_{N,\tau}=\delta(V)$ [from
	\eqref{S=delta(v)}], whence $\E \< D\mathcal{S}_{N,t}\,, v_{N,\tau}\>_{\HH}
	=\Cov(\mathcal{S}_{N,t}\,,\mathcal{S}_{N,\tau})$.  Since the Walsh integral
	has mean zero and $U$ is adapted, $\E\mathcal{Y}_{N,t,\tau}=0$; see
	\eqref{XY}. This and \eqref{DS:v:X:Y} together complete the proof.
\end{proof}

\begin{lemma}\label{lem:Var(X)}
	For every $T>0$ there exists a real number $A_T>0$ such that
	\[
		\sup_{t,\tau\in(0,T)}\Var(\mathcal{X}_{N,t,\tau}) \le A_T\,\frac{(\log N)^3}{N^3}
		\quad\text{uniformly for every $N\ge\e$}.
	\]
\end{lemma}

\begin{proof}
	Choose and fix $0<t,\tau<T$ and $N\ge\e$.
	It follows readily from \eqref{XY} and our efforts thus far that $\mathcal{X}_{N,t,\tau}$
	is Malliavin differentiable, and the following is a version of the Malliavin derivative:
	\[
		D_{r,z}\mathcal{X}_{N,t,\tau} = 2\bm{1}_{[0,t\wedge\tau]}(r)\int_r^{t\wedge\tau}
		\d s\int_{-\infty}^\infty\d y\
		g_{N,t}(s\,,y) g_{N,\tau}(s\,,y)U(s\,,y) D_{r,z}U(s\,,y).
	\]
	Moreover, it follows from this and the definition of the $\mathcal{H}$-norm that
	\[
		\| D\mathcal{X}_{N,t,\tau}\|_{\HH}^2 = 4\int_0^{t\wedge\tau}\d r\int_{-\infty}^\infty\d z
		\left|\int_r^{t\wedge\tau}
		\d s\int_{-\infty}^\infty\d y\
		g_{N,t}(s\,,y) g_{N,\tau}(s\,,y)U(s\,,y) D_{r,z}U(s\,,y)\right|^2.
	\]
	According to \eqref{U}, \eqref{||u(s,y)||}, and Lemma \ref{derivative:estimate},
	whenever $0<s,s'<T$ and $y,y'\in\R$, the following holds a.s.\
	for a.e.\ every $(r\,,z)\in(s\wedge s',t)\times\R$:
	\begin{align*}
		\left| \E\left[ U(s\,,y) D_{r,z}U(s\,,y) U(s',y') D_{r,z}U(s',y')\right]\right|
			&\le c_{4,T}^2\left\| D_{r,z}U(s\,,y)\right\|_4 \left\|
			D_{r,z}U(s',y')\right\|_4\\
		&\le c_{T,4}^2C_{T,4}^2\,\frac{\bm{p}_{s-r}(y-z)\bm{p}_r(z)}{\bm{p}_s(y)}
			\frac{\bm{p}_{s'-r}(y'-z)\bm{p}_r(z)}{\bm{p}_{s'}(y')}\\
		&=: \tfrac14 A_T\,
			\bm{p}_{r(s-r)/s}\left( z - \frac rs y\right)
			\bm{p}_{r(s'-r)/s'}\left( z - \frac{r}{s'} y'\right),
	\end{align*}
	where we have appealed to \eqref{PPPP} in the last line. Therefore,
	\begin{align*}
		\E\left(\|D\mathcal{X}_{N,t,\tau}\|^2_{\HH}\right) &\le A_T\int_0^{t\wedge\tau}\d r\int_{-\infty}^\infty\d z
			\int_r^{t\wedge\tau}
			\d s\int_{-\infty}^\infty\d y\int_r^{t\wedge\tau}
			\d s'\int_{-\infty}^\infty\d y'\\
		&\quad\times g_{N,t}(s\,,y) g_{N,\tau}(s\,,y)g_{N,t}(s',y') g_{N,\tau}(s',y')
			\bm{p}_{r(s-r)/s}\left( z - \frac rs y\right)
			\bm{p}_{r(s'-r)/s'}\left( z - \frac{r}{s'} y'\right)\\
		&= A_T\int_0^{t\wedge\tau}\d r
			\int_r^{t\wedge\tau}
			\d s\int_{-\infty}^\infty\d y\int_r^{t\wedge\tau}
			\d s'\int_{-\infty}^\infty\d y'\\
		&\quad\times g_{N,t}(s\,,y) g_{N,\tau}(s\,,y)g_{N,t}(s',y') g_{N,\tau}(s',y')
			\bm{p}_{[r(s-r)/s]+[r(s'-r)/s']}\left( \frac rs y - \frac{r}{s'}y'\right),
	\end{align*}
	thanks to the semigroup property of the heat kernel. Since $g_{N,\nu}(s,
	y)\le \frac{\nu}{s}N^{-1}$ for all $N>0,\nu \geq s>0$ and $y\in \R$, we may
	bound two of the $g$-terms from above, each by $N^{-1}$, in order to find
	that
	\begin{align*}
		\E\left(\|D\mathcal{X}_{N,t,\tau}\|^2_{\HH}\right) &\le \frac{A_T}{N^2}\int_0^{t\wedge\tau}\d r
			\int_r^{t\wedge\tau}
			\frac{{t\vee\tau}}{s}\d s\int_{-\infty}^\infty\d y\int_r^{t\wedge\tau}
			\frac{{t\vee\tau}}{s'}\d s'\int_{-\infty}^\infty\d y'\\
		&\quad\times g_{N,t\wedge\tau}(s\,,y)g_{N,t\wedge\tau}(s',y')
			\bm{p}_{[r(s-r)/s]+[r(s'-r)/s']}\left( \frac rs y - \frac{r}{s'}y'\right)\\
		&=\frac{A_T}{N^4}\int_0^{t\wedge\tau}\d r
			\int_r^{t\wedge\tau}
			\frac{{t\vee\tau}}{s}\d s\int_{-\infty}^\infty\d y\int_r^{t\wedge\tau}
			\frac{{t\vee\tau}}{s'}\d s'\int_{-\infty}^\infty\d y'\int_0^N\d x\int_0^N\d x'\\
		&\quad\times \bm{p}_{s(\{t\wedge\tau\}-s)/(t\wedge\tau)}\left( y - \frac{s}{t\wedge\tau} x\right)
			\bm{p}_{s'(\{t\wedge\tau\}-s')/(t\wedge\tau)}\left( y' - \frac{s'}{t\wedge\tau} x'\right)\\
		&\quad\times \bm{p}_{[r(s-r)/s]+[r(s'-r)/s']}\left( \frac rs y - \frac{r}{s'}y'\right).
	\end{align*}
	It follows from \eqref{p-scale} that
	\[
		\bm{p}_{[r(s-r)/s]+[r(s'-r)/s']}\left( \frac rs y - \frac{r}{s'}y'\right)
		=\frac sr \bm{p}_{[s(s-r)/r]+[s^2(s'-r)/(s'r)]}\left( y - \frac{s}{s'}y'\right).
	\]
	Therefore, the semigroup property of the heat kernel implies the following:
	\begin{align*}
		\E\left(\|D\mathcal{X}_{N,t,\tau}\|^2_{\HH}\right) &\le \frac{A_T(t\vee \tau)^2}{N^4}\int_0^{t\wedge\tau}\frac{\d r}{r}
			\int_r^{t\wedge\tau} \d s\int_r^{t\wedge\tau}
			\frac{\d s'}{s'}\int_{-\infty}^\infty\d y'\int_0^N\d x\int_0^N\d x'\\
		&\hskip.6in \bm{p}_{s'(\{t\wedge\tau\}-s')/(t\wedge\tau)}\left( y' - \frac{s'}{t\wedge\tau} x'\right)\\
		&\hskip.6in \times
			\bm{p}_{[s(s-r)/r]+[s^2(s'-r)/(s'r)]+[s(\{t\wedge\tau\}-s)/(t\wedge\tau)]}
			\left( \frac{s}{s'}y' - \frac{s}{t\wedge\tau}x\right).
	\end{align*}
	A repeated appeal to \eqref{p-scale} yields
	\begin{align*}
			&\bm{p}_{[s(s-r)/r]+[s^2(s'-r)/(s'r)]+[s(\{t\wedge\tau\}-s)/(t\wedge\tau)]}
				\left( \frac{s}{s'}y' - \frac{s}{t\wedge\tau}x\right)\\
			&\hskip1in
				=\frac{s'}{s}\bm{p}_{[(s')^2(s-r)/(sr)]+[s'(s'-r)/r]+[(s')^2(\{t\wedge\tau\}-s)/\{s(t\wedge\tau)\}]}
				\left( y' - \frac{s'}{t\wedge\tau}x\right).
	\end{align*}
	And yet another appeal to the semigroup property reveals the following:
	\begin{align*}
		&\E\left(\|D\mathcal{X}_{N,t,\tau}\|^2_{\HH}\right)
			\le \frac{A_T (t\vee\tau)^2}{N^4}\int_0^{t\wedge\tau}\frac{\d r}{r}
			\int_r^{t\wedge\tau} \frac{\d s}{s}\int_r^{t\wedge\tau} \d s'\int_0^N\d x\int_0^N\d x'\\
		& \quad\times\bm{p}_{[(s')^2(s-r)/(sr)]+[s'(s'-r)/r]+[(s')^2(\{t\wedge\tau\}-s)/\{s(t\wedge\tau)\}]
			+[s'(\{t\wedge\tau\}-s')/(t\wedge\tau)]}
			\left( \frac{s'}{t\wedge\tau}(x-x')\right)\\
		&=\frac{A_T(t\vee\tau)^2(t\wedge \tau)}{N^4}
			\int_0^{t\wedge\tau} \frac{\d s}{s}\int_0^{t\wedge\tau}\frac{\d s'}{s'}
			\int_0^{s\wedge s'} \frac{\d r}{r}
			\int_0^N\d x\int_0^N\d x'\\
		& \quad\times
			\bm{p}_{[(t\wedge\tau)^2(s-r)/(sr)]+[(t\wedge\tau)^2(s'-r)/(s'r)]+[(t\wedge\tau)(\{t\wedge\tau\}-s)/s]
			+[(t\wedge\tau)(\{t\wedge\tau\}-s')/s']}(x-x'),
	\end{align*}
	thanks also to scaling \eqref{p-scale} and Fubini's theorem. Since
	\begin{align*}
		\frac{2(t\wedge \tau)(t\wedge \tau -r)}{r}= \frac{(t\wedge\tau)^2(s-r)}{sr}+
		\frac{(t\wedge\tau)^2(s'-r)}{s'r}+
		\frac{(t\wedge\tau)(t\wedge\tau-s)}{s} +
		\frac{(t\wedge\tau)(t\wedge\tau-s')}{s'},
	\end{align*}
	we appeal to Lemma \ref{lem7.2} in order to find that
	\[
		\E\left(\|D\mathcal{X}_{N,t,\tau}\|^2_{\HH}\right)\le
		\frac{A_T(t\vee\tau)^{2}(t\wedge \tau)}{N^3\pi}
		\int_0^{t\wedge\tau} 			\frac{\d s}{s}\int_0^{t\wedge\tau}\frac{\d s'}{s'}
		\int_0^{s\wedge s'}\frac{\d r}{r}
		\int_{-\infty}^\infty\d z\ \varphi(z)\e^{-((t\wedge \tau)((t\wedge \tau)-r))z^2/(rN^2)}.
	\]
	Integrating in the variables $s$ and $ s'$ yields
	\[
	\E\left(\|D\mathcal{X}_{N,t,\tau}\|^2_{\HH}\right)\le
		\frac{A_T(t\vee\tau)^{2}(t\wedge \tau)}{N^3\pi}
		\int_0^{t\wedge \tau}\frac{\d r}{r}\left(\log \left(\frac {t\wedge\tau}{r}\right)\right)^2 \int_{\R} \e^{-\frac { (t\wedge \tau)((t\wedge \tau)-r)}r \frac { z^2} {N^2}}  \varphi(z) \d z,
	\]
    Making the change of variables $\frac {(t\wedge\tau)-r} r =\theta$, allows us to write
	\[
		\E\left(\|D\mathcal{X}_{N,t,\tau}\|^2_{\HH}\right)\le
		\frac{A_T(t\vee\tau)^{2}(t\wedge \tau)}{N^3\pi} \int_{\R}  \varphi(z) \d z
		\int_0^\infty   \d \theta \frac 1{\theta +1}  \left( \log (\theta +1) \right)^2  \e^{-\frac { (t\wedge \tau)\theta  z^2} {N^2}}.
	\]
	Integrating by parts and using the fact that
	\[
		\left(\frac 13 ( \log(\theta+1))^3    \e^{-\frac { (t\wedge \tau)\theta  z^2} {N^2}}  \right)  _{\theta=0} ^{\theta=\infty} =0,
	\]
	we obtain
  \begin{align*}
		\E\left(\|D\mathcal{X}_{N,t,\tau}\|^2_{\HH}\right)&\le \frac{A_T(t\vee\tau)^{2}(t\wedge \tau)}{3N^3\pi}
		\int_{\R}   \varphi(z) \d z
		\int_0^\infty   \d \theta   \left( \log (\theta +1) \right)^3  \e^{-\frac { t\theta  z^2} {N^2}}  \frac {(t\wedge \tau)z^2}{N^2}  \\
		&  =		\frac{A_T(t\vee\tau)^{2}(t\wedge \tau)}{3N^3\pi}
		\int_{\R} \varphi(z) \d z \int_0^\infty \d \theta   \left( \log \left(\frac {N^2} {(t\wedge\tau)z^2}\theta +1\right) \right)^3     \e^{-\theta}.
	\end{align*}
	Using the inequality
	\begin{align*}
		\log \left(\frac {N^2} {(t\wedge\tau)z^2}\theta +1\right) &\le
		2 \log N + \log(\theta+1) + \log \left( \frac 1{t\wedge\tau}+ 1\right) + \log \left( \frac 1 {z^2} +1\right) \\
		&\le  \left( 2 \log N + \log \left( \frac 1{t\wedge\tau}+ 1\right)\right) \left(1+ \log(\theta+1)+ \log \left( \frac 1 {z^2} +1\right)  \right),
	\end{align*}
  and taking into account that
  \[
	 C:=  \int_{\R} \varphi(z) \d z
      \int_0^\infty   \d \theta   \left(1+ \log(\theta+1)+ \log \left( \frac 1 {z^2} +1\right)  \right)^3 e^{-\theta}<\infty,
	\]
	we finally get
	\[
		\E\left(\|D\mathcal{X}_{N,t,\tau}\|^2_{\HH}\right)\le
		\frac{CA_T(t\vee\tau)^{2}(t\wedge \tau)}{3N^3\pi} \left( 2 \log N + \log\left( \frac 1{t\wedge\tau}+ 1\right)\right)^3,
  \]
  which provides the desired estimate.
\end{proof}

{
\begin{lemma}\label{lem:Var(Y)}
	For every $T>0$ there exists a real number $A_T'>0$ such that
	\[
		\sup_{t,\tau\in(0,T)}\Var\left(\mathcal{Y}_{N,t,\tau}\right) \le
		A_T'\,\frac{(\log N)^3}{N^3}
		\quad\text{uniformly for every $N\ge\e$}.
	\]
\end{lemma}

\begin{proof}
	Lemma \ref{lem:E<DS,v>} ensures that $\mathcal{Y}_{N,t,\tau}$ has mean zero, and hence
	\begin{align*}
		\Var\left(\mathcal{Y}_{N,t,\tau}\right)=\E\int_0^{t\wedge \tau}\d r
		\int_{-\infty}^\infty\d z\int_0^{t\wedge\tau}\d r'
		\int_{-\infty}^\infty\d z'
		\left( \int_{(r,t)\times\R} v_{N,\tau}(r\,,z)D_{r,z} v_{N,t}(s\,,y)
		\,\eta(\d s\,\d y)\right)&\\
		\times \left( \int_{(r',t)\times\R} v_{N,\tau}(r',z')
		D_{r',z'} v_{N,t}(s\,,y)\,\eta(\d s\,\d y)\right)&,
	\end{align*}
	which, by Fubini's theorem, is
	\begin{align*}
		= \int_0^{t\wedge\tau}\d r\int_{-\infty}^\infty\d z
		\int_0^{t\wedge\tau}\d r'\int_{-\infty}^\infty\d z'
		\int_{r\vee r'}^{t}\d s\int_{-\infty}^\infty\d y
		\: g_{N,\tau}(r\,,z) g_{N,\tau}(r',z')
		g^2_{N,t}(s\,,y) &\\
		\times \E\left[U(r\,,z)\cdot D_{r,z}U(s,y) \cdot U(r'\,,z')
	\cdot D_{r',z'} U(s\,,y)\right]&.
	\end{align*}
	Combine \eqref{U} and \eqref{||u(s,y)||} with Lemma \ref{derivative:estimate}
	in order to see that
	\begin{align*}
		&\left| \E\left[U(r\,,z)\cdot D_{r,z} U(s\,,y)\cdot
		U(r'\,,z')\cdot D_{r',z'} U(s\,,y)\right] \right|
			\le c_{T,4}^2
			\| D_{r,z} U(s\,,y)\|_4 \|D_{r',z'} U(s\,,y)\|_4\\
		&\le c_{T,4}^2C_{T,4}^2\,\frac{\bm{p}_{s-r}(y-z)\bm{p}_r(z)}{\bm{p}_s(y)}
			\frac{\bm{p}_{s-r'}(y-z')\bm{p}_{r'}(z')}{\bm{p}_s(y)}
			=: L_T\,
			\bm{p}_{r(s-r)/s}\left( z - \frac rs y\right)
			\bm{p}_{r'(s-r')/s}\left( z' - \frac{r'}{s} y\right).
	\end{align*}
	Plug this into the preceding identity for $\Var(\mathcal{Y}_{N,t,\tau})$ in order to see that
	\begin{align*}
		&\Var(\mathcal{Y}_{N,t,\tau})
			\le L_T\int_0^{t\wedge\tau}\d r\int_{-\infty}^\infty\d z
			\int_0^{t\wedge\tau}\d r'\int_{-\infty}^\infty\d z'
			\int_{r\vee r'}^{t}\d s\int_{-\infty}^\infty\d y\\
		&\hskip.9in\times g_{N,\tau}(r\,,z) g_{N,\tau}(r',z')
			g^2_{N,t,x}(s\,,y)
			\bm{p}_{r(s-r)/s}\left( z - \frac rs y\right)
			\bm{p}_{r'(s-r')/s}\left( z' - \frac{r'}{s} y\right).
	\end{align*}
	We can apply first \eqref{g+v}, and then the semigroup property of the heat kernel,
	in order to see that for all $\tilde{x}\in [0, 1]$
	\begin{align*}
		\int_{-\infty}^\infty g_{N,\tau}(r\,,z)\bm{p}_{r(s-r)/s}\left( z - \frac rs y\right)\d z
			&\leq\frac1N\int_0^N\d x_1\int_{-\infty}^\infty\d z\
			\bm{p}_{r(s-r)/s}\left( z - \frac rs y\right)\bm{p}_{r(\tau-r)/\tau}\left( z - \frac r\tau x_1\right)\\
		&=\frac1N\int_0^N \bm{p}_{[r(s-r)/s] + [r(\tau-r)/\tau]}\left( \frac rs y - \frac r\tau x_1\right)\d x_1.
	\end{align*}
	Therefore,
	\begin{align*}
		\Var(\mathcal{Y}_{N,t, \tau})
		&\le \frac{L_T}{N^2}\int_0^{t\wedge\tau}\d r\int_0^{t\wedge\tau}\d r'
			\int_{r\vee r'}^{t}\d s\int_{-\infty}^\infty\d y\int_0^N\d x_1\int_0^N\d x_2\: g^2_{N,t}(s\,,y) \\
		&\qquad\times
			\bm{p}_{[r(s-r)/s] + [r(\tau-r)/\tau]}\left( \frac rs y - \frac r\tau x_1\right)
			\bm{p}_{[r'(s-r')/s] + [r'(\tau-r')/\tau]}\left( \frac{r'}{s} y - \frac{r'}{\tau} x_2\right).
	\end{align*}
	Since  $g_{N,\nu}(s, y)\le \frac{\nu}{s}N^{-1}$
	for all $N>0,\nu\geq s>0$ and $y\in \R$, we have
	\begin{align*}
		\Var(\mathcal{Y}_{N,t, \tau})
			&\le \frac{t^2 L_T}{N^4}\int_0^{t\wedge\tau}\d r\int_0^{t\wedge\tau}\d r'
			\int_{r\vee r'}^{t}\frac{\d s}{s^2}\int_{-\infty}^\infty\d y\int_0^N\d x_1\int_0^N\d x_2\\
		&\quad\times
			\bm{p}_{[r(s-r)/s] + [r(\tau-r)/\tau]}\left( \frac rs y - \frac r\tau x_1\right)
			\bm{p}_{[r'(s-r')/s] + [r'(\tau-r')/\tau]}\left( \frac{r'}{s} y - \frac{r'}{\tau} x_2\right).
	\end{align*}
	Now we use scaling [see \eqref{p-scale}] to see that
	\[
			\bm{p}_{[r(s-r)/s] + [r(\tau-r)/\tau]}\left( \frac rs y - \frac r\tau x_1\right) =
			\frac sr\,\bm{p}_{[s(s-r)/r] + [s^2(\tau-r)/(r\tau)]}\left( y - \frac s\tau x_1\right),
	\]
	with an analogous expression holding for the version with the variables with the primes. This
	endeavor, and the semigroup property of the heat kernel, together yield
	\[
		\Var(\mathcal{Y}_{N,t,\tau})
		\le \frac{t^2 L_T}{N^4}\int_0^{t\wedge\tau}\frac{\d r}{r}\int_0^{t\wedge\tau}\frac{\d r'}{r'}
		\int_{r\vee r'}^{t}\d s\int_0^N\d x_1\int_0^N\d x_2\
		\bm{p}_{_{\Gamma+\Gamma'}}\left( \frac{s}{\tau} x_1 - \frac{s}{\tau} x_2\right),
	\]
	with $\Gamma$ and $\Gamma'$
	being the following functions whose variable-dependencies are excised for  ease
	of exposition:
	\[
		\Gamma := \frac{s(s-r)}{r} + \frac{s^2(\tau-r)}{r\tau},
		\qquad
		\Gamma' := \frac{s(s-r')}{r'} + \frac{s^2(\tau-r')}{r'\tau}.
	\]
	A change of variables $[a=sx_1/\tau,\, a'=sx_2/\tau$] yields
	\begin{align}
		\Var(\mathcal{Y}_{N,t,\tau})
			&\le \frac{t^2\tau^2 L_T}{N^4}\int_0^{t\wedge\tau}\frac{\d r}{r}\int_0^{t\wedge\tau}\frac{\d r'}{r'}
			\int_{r\vee r'}^{t}\d s\ s^{-2}\int_0^{Ns/\tau}\d a\int_0^{Ns/\tau}\d a'\
			\bm{p}_{_{\Gamma+\Gamma'}}(a-a')\nonumber\\
		&=\frac{t^2\tau L_T}{\pi N^3}
			\int_0^{t}\frac{\d s}{s}
			\int_0^{s\wedge\tau}\frac{\d r}{r}\int_0^{s\wedge\tau}\frac{\d r'}{r'}
			\int_{-\infty}^\infty\d z\
			\varphi(z){\e^{-(\Gamma+\Gamma')z^2 \tau^2/(2N^2s^2)}}\nonumber\\
		&=\frac{t^2\tau L_T}{\pi N^3}\int_{-\infty}^\infty\d z\ \varphi(z)
			\int_0^{t}\frac{\d s}{s}
			\left(\int_0^{s\wedge\tau}\frac{\d r}{r}
			{\e^{-\Gamma z^2 \tau^2/(2N^2s^2)}}\right)^2\nonumber\\
			&=\frac{t^2\tau L_T}{\pi N^3}\int_{-\infty}^\infty\d z\ \varphi(z)
			\int_0^{t}\frac{\d s}{s}
			\left(\int_0^{s\wedge\tau}\frac{\d r}{r}
			{\e^{-z^2(r^{-1}-s^{-1}+r^{-1}-\tau^{-1}) \tau^2/(2N^2)}}\right)^2, \label{case2}
	\end{align}
	where $\varphi(z)=(1-\cos z)/z^2$,
	and we have used Lemma \ref{lem7.2} in the first equality. 
	
	\textbf{Case 1: $t\leq \tau$}.  In this case, from the proceeding,
	\begin{align*}
		\Var(\mathcal{Y}_{N,t,\tau})
		&\leq\frac{t^2\tau L_T}{\pi N^3}\int_{-\infty}^\infty\d z\ \varphi(z)
			\int_0^{\tau}\frac{\d s}{s}
			\left(\int_0^{s}\frac{\d r}{r}
			{\e^{-z^2(r^{-1}-s^{-1}+r^{-1}-\tau^{-1}) \tau^2/(2N^2)}}\right)^2.
	\end{align*}
	Using change of variable $\frac{s-r}{r}=\theta$, we obtain
	\begin{align*}
		\Var(\mathcal{Y}_{N,t,\tau})
		&\leq\frac{t^2\tau L_T}{\pi N^3}\int_{-\infty}^\infty\d z\ \varphi(z)
			\int_0^{\tau}\frac{\d s}{s}
			\left(\int_0^{\infty}\frac{\d \theta}{\theta+1}
			{\e^{-z^2(2\theta s^{-1}+s^{-1}-\tau^{-1}) \tau^2/(2N^2)}}\right)^2\\
		&=\frac{t^2\tau L_T}{\pi N^3}\int_{-\infty}^\infty\d z\ \varphi(z)
			\int_0^{\tau}\frac{\d s}{s} {\e^{-z^2(s^{-1}-\tau^{-1}) \tau^2/N^2}}
			\left(\int_0^{\infty}\frac{\d \theta}{\theta+1}
			{\e^{-z^2\theta\tau^2/(sN^2)}}\right)^2\\
		&=\frac{t^2\tau L_T}{\pi N^3}\int_{-\infty}^\infty\d z\ \varphi(z)
			\int_0^{\infty}\frac{\d \xi}{\xi +1} {\e^{-z^2\xi\tau/N^2}}
			\left(\int_0^{\infty}\frac{\d \theta}{\theta+1}
			{\e^{-z^2\theta(\xi+1)\tau/N^2}}\right)^2\\
		&\leq\frac{t^2\tau L_T}{\pi N^3}\int_{-\infty}^\infty\d z\ \varphi(z)
			\left(\int_0^{\infty}\frac{\d \theta}{\theta+1}
			{\e^{-z^2\theta\tau/N^2}}\right)^3,
	\end{align*}
	where, in the second equality, we use change of variable $\frac{\tau-s}{s}=\xi$.
	  Since for $z\neq0$
  \begin{align}
		\int_0^{\infty}\frac{\d \theta}{\theta+1}
			{\e^{-z^2\theta\tau/N^2}}&= \int_0 ^\infty \frac 1 {\theta+  \frac {\tau z^2}{N^2}}     \e^{-\theta}   \d \theta \nonumber\\
			 &
		\le
		\int_1 ^\infty     \e^{-\theta}   \d \theta +
		\int_0 ^1  \frac 1 {\theta+  \frac {\tau z^2}{N^2}}         \d \theta =  \e^{-1} +
		\log \left( 1+ \frac {N^2} {\tau z^2} \right)    \nonumber \\
		& \le \e^{-1} +  2 \log N + \log (1+ 1/\tau)+ \log (1+  z^{-2}) \label{log}
	\end{align}
	and taking into account that
	\[
		\int_{\R}\varphi(z)(1 + \log (1+  z^{-2}) )^3 \d z <\infty,
	\]
	we obtain that
		\[
		\sup_{\substack{0<t\leq\tau\leq T}}\Var\left(\mathcal{Y}_{N,t,\tau}\right) \le
		L_T'\,\frac{(\log N)^3}{N^3}
		\quad\text{uniformly for every $N\ge\e$}.
	\]
	
		\textbf{Case 2: $t> \tau$}.  In this case, according to \eqref{case2}, we have
	\begin{align}
		\Var(\mathcal{Y}_{N,t,\tau})
		&\leq\frac{t^2\tau L_T}{\pi N^3}\int_{-\infty}^\infty\d z\ \varphi(z)
			\int_0^{\tau}\frac{\d s}{s}
			\left(\int_0^{s}\frac{\d r}{r}
			{\e^{-z^2(r^{-1}-s^{-1}+r^{-1}-\tau^{-1}) \tau^2/(2N^2)}}\right)^2 \nonumber\\
			&\quad + \frac{t^2\tau L_T}{\pi N^3}\int_{-\infty}^\infty\d z\ \varphi(z)
			\int_{\tau}^t\frac{\d s}{s}
			\left(\int_0^{\tau}\frac{\d r}{r}
			{\e^{-z^2(r^{-1}-s^{-1}+r^{-1}-\tau^{-1}) \tau^2/(2N^2)}}\right)^2\nonumber\\
			& \leq \frac{L_T'(\log N)^3}{N^3} + \frac{t^2\tau L_T}{\pi N^3}\int_{-\infty}^\infty\d z\ \varphi(z)
			\int_{\tau}^t\frac{\d s}{s}
			\left(\int_0^{\tau}\frac{\d r}{r}
			{\e^{-z^2(r^{-1}-s^{-1}+r^{-1}-\tau^{-1}) \tau^2/(2N^2)}}\right)^2 \label{t>tau}
	\end{align}
	where the second inequality holds by the result of \textbf{Case 1}.  Moreover, a change of variable $\frac{\tau-r}{r}=\theta$ yields that
	\begin{align*}
	&\int_{-\infty}^\infty\d z\ \varphi(z)
			\int_{\tau}^t\frac{\d s}{s}
			\left(\int_0^{\tau}\frac{\d r}{r}
			{\e^{-z^2(r^{-1}-s^{-1}+r^{-1}-\tau^{-1}) \tau^2/(2N^2)}}\right)^2\\
	&\quad =\int_{-\infty}^\infty\d z\ \varphi(z)
			\int_{\tau}^t\frac{\d s}{s}
			\left(\int_0^{\infty}\frac{\d \theta}{1+\theta}
			{\e^{-z^2(\tau^{-1}-s^{-1}+2\theta\tau^{-1}) \tau^2/(2N^2)}}\right)^2\\
	&\quad \leq \int_{-\infty}^\infty\d z\ \varphi(z)
			\int_{\tau}^t\frac{\d s}{s}
			\left(\int_0^{\infty}\frac{\d \theta}{1+\theta}
			{\e^{-z^2\theta\tau/N^2}}\right)^2\\
	&\quad \leq \log\frac{t}{\tau}\int_{-\infty}^\infty\d z\ \varphi(z)
			\left(\e^{-1} +  2 \log N + \log (1+ 1/\tau)+ \log (1+  z^{-2})\right)^2,
	\end{align*}
	where the last inequality is due to \eqref{log}. The proceeding together with \eqref{t>tau} implies that 
			\[
		\sup_{\substack{0<\tau<t\leq T}}\Var\left(\mathcal{Y}_{N,t,\tau}\right) \le
		L_T''\,\frac{(\log N)^3}{N^3}
		\quad\text{uniformly for every $N\ge\e$}.
	\]
	
	The proof is complete. 
\end{proof}
}

We now conclude this section with the following.

\begin{proof}[Proof of Theorem \ref{TVD}]
	From Proposition \eqref{Var<Ds,v>} [with $t=\tau$], we
	see that for all $T>0$ there exists a number
	$K_T>0$ such that
	\[
		\Var\left< D\mathcal{S}_{N,t}\,,v_{N,t}\right>_\HH \le K_T
		\frac{(\log N)^3}{N^3}
		\quad\text{for all $t\in(0\,,T)$ and $N\ge\e$.}
	\]
	By \eqref{S=delta(v)} and Proposition \ref{pr:MS},
	\begin{align*}
		d_{\rm TV}\left( \frac{\mathcal{S}_{N,t}}{\sqrt{\Var(\mathcal{S}_{N,t})}}
			~,~ Z\right) &\le 2\sqrt{\Var\left< \frac{D\mathcal{S}_{N,t}}{\sqrt{\Var(\mathcal{S}_{N,t})}}
			~,~ \frac{v_{N,t}}{\sqrt{\Var(\mathcal{S}_{N,t})}}\right>_{\HH}
			}\\
		&\le 2\sqrt{K_T}\,
		\frac{(\log N)^{3/2}}{N^{3/2}\Var(\mathcal{S}_{N,t})}\qquad\text{uniformly for all
			$t\in(0\,,T)$ and $N\ge\e$.}
	\end{align*}
	 Proposition \ref{pr:Cov:asymp} ensures that
	$\Var(\mathcal{S}_{N,t})\sim 2t\log(N)/N$ as $N\to\infty$,
	which concludes the proof.
\end{proof}

\section{Proof of Theorem \ref{th:FCLT}} \label{sec:6}

In order to prove Theorem \ref{th:FCLT} we need to establish the weak
convergence of the finite-dimensional distributions, as well as tightness.  The
following addresses tightness.

\begin{proposition}[Tightness]\label{pr:tightness}
	For every $T>0$, $k\ge2$, and $\gamma\in(0\,,1/6)$,
	there exists a number $L=L(T,k\,,\gamma)>0$
	such that for all $\varepsilon \in (0\,,1]$,
	\[
		\sup_{0<t\le T}
		\E\left(  \left| \mathcal{S}_{N,t+\varepsilon}-\mathcal{S}_{N,t}
		\right|^k \right) \le L \varepsilon^{\gamma k}
		\left(\frac{\log N}{N}\right)^{k/2}\qquad\text{uniformly for all $N\ge\e$.}
	\]
\end{proposition}

The proof of Proposition \ref{pr:tightness} hinges on the following lemma,
which is a useful inequality when $t$ stays away from zero.

\begin{lemma}\label{lem:tightness}
	For every $T>0$, $k\ge2$ and $\delta>0$,
	there exists a number $K=K(T,k, \delta)>0$ such that
	\[
		\E\left( \left| \mathcal{S}_{N,t+\varepsilon}-\mathcal{S}_{N,t}
		\right|^k \right)\le \frac{K\varepsilon^{ k/2}}{(t\wedge 1)^{k(1+\delta)/2}}
		\left(\frac{\log N}{N}\right)^{k/2},
	\]
	uniformly for all $N\ge\e$, $t\in(0\,,T]$, and $\epsilon\in(0,1)$.
\end{lemma}

\begin{proof}
	Thanks to \eqref{E:U} and \eqref{average}, we may write
	the following: For all $N,t>0$,
	\begin{align*}
		\mathcal{S}_{N,t+\varepsilon} - \mathcal{S}_{N,t} &=
			\frac1N\int_0^N \left[ U(t+\varepsilon\,,x) - U(t\,,x)\right]\d x\\
		&= \int_{(0,t)\times\R} U(s\,,y)\mathcal{A}(s\,,y)\,\eta(\d s\,\d y) +
			\int_{(t,t+\varepsilon)\times\R} U(s\,,y)\mathcal{B}(s\,,y) \,\eta(\d s\,\d y),
	\end{align*}
	almost surely, where
	\begin{align*}
		\mathcal{A}(s\,,y) &:= \frac1N\int_0^N\left[
			\bm{p}_{s(t+\varepsilon-s)/(t+\varepsilon)}\left( y - \frac{sx}{t+\varepsilon}\right)
			- \bm{p}_{s(t-s)/t}\left( y - \frac{sx}{t}\right)
			\right]\d x,\qquad\text{and}\\
		\mathcal{B}(s\,,y) &:= \frac1N\int_0^N
			\bm{p}_{s(t+\varepsilon-s)/(t+\varepsilon)}\left( y - \frac{sx}{t+\varepsilon}\right)
			\d x,
	\end{align*}
	and the dependence on the parameters $N$ and $\varepsilon$ are subsumed for ease of
	notation. Thus,
	\begin{equation}\label{STAB}
		\|\mathcal{S}_{N,t+\varepsilon} - \mathcal{S}_{N,t}\|_k \le
		T_{\mathcal{A}} + T_{\mathcal{B}},
	\end{equation}
	where
	\[
		T_{\mathcal{A}} :=
		\left\| \int_{(0,t)\times\R} U(s\,,y)
		\mathcal{A}(s\,,y) \,\eta(\d s\,\d y)\right\|_k
		\quad\text{and}\quad
		T_{\mathcal{B}} :=
		\left\| \int_{(t,t+\varepsilon)\times\R} U(s\,,y)
		\mathcal{B}(s\,,y) \,\eta(\d s\,\d y)\right\|_k.
	\]
	We will estimate $T_{\mathcal{A}}$ and $T_{\mathcal{B}}$
	separately and  in reverse order.

	To estimate $T_{\mathcal{B}}$ we appeal to the BDG inequality (with BDG constant
	$c_k$) as follows:
	\begin{align*}
		T_{\mathcal{B}}^2 &\le c_k\int_t^{t+\varepsilon}\d s\int_{-\infty}^\infty\d y\
			\| U(s\,,y)\|_k^2|\mathcal{B}(s\,,y)|^2
			\le c_k c_{k,T}^2\int_t^{t+\varepsilon}\d s\int_{-\infty}^\infty\d y\
			|\mathcal{B}(s\,,y)|^2\\
		&= \frac{c_kc_{k,T}^2}{N^2}\int_t^{t+\varepsilon}\d s\int_{-\infty}^\infty\d y
			\int_0^N\d x_1\int_0^N\d x_2\
			\bm{p}_{s(t+\varepsilon-s)/(t+\varepsilon)}\left( y - \frac{sx_1}{t+\varepsilon}\right)
			\bm{p}_{s(t+\varepsilon-s)/(t+\varepsilon)}\left( y - \frac{sx_2}{t+\varepsilon}\right),
	\end{align*}
	where we used \eqref{||u(s,y)||} to deduce the second inequality. Rearrange the integrals
	and compute the $\d y$-integral first to see from the semigroup property of the
	heat kernel that
	\begin{align*}
		T_{\mathcal{B}}^2 &\le \frac{c_kc_{k,T}^2}{N^2}\int_t^{t+\varepsilon}\d s
			\int_0^N\d x_1\int_0^N\d x_2\
			\bm{p}_{2s(t+\varepsilon-s)/(t+\varepsilon)}\left( \frac{s(x_1-x_2)}{t+\varepsilon}\right)\\
		&= \frac{c_kc_{k,T}^2(t+\varepsilon)^2}{N^2}\int_t^{t+\varepsilon}
			\frac{\d s}{s^2}\int_0^{sN/(t+\varepsilon)}\d x_1\int_0^{sN/(t+\varepsilon)}\d x_2\
			\bm{p}_{2s(t+\varepsilon-s)/(t+\varepsilon)}(x_1-x_2),
	\end{align*}
	after a change of variables. Since the $\d x_2$-integral is bounded above by one,
	it follows that
	\begin{equation}\label{T_B}
		T_{\mathcal{B}}^2 \le \frac{c_kc_{k,T}^2(t+\varepsilon)}{N}\int_t^{t+\varepsilon}
		\frac{\d s}{s}< \frac{c_kc_{k,T}^2(t+\varepsilon)}{Nt}\,\varepsilon.
	\end{equation}

	The estimation of $T_{\mathcal{A}}$ is more involved, though it starts in the same way as
	did the process of bounding $T_{\mathcal{B}}$. Namely, we write, using the BDG inequality,
	\begin{equation}\label{pre:T_A}\begin{split}
		T_{\mathcal{A}}^2  &\le c_k\int_0^t\d s\int_{-\infty}^\infty\d y\
			\| U(s\,,y)\|_k^2|\mathcal{A}(s\,,y)|^2\\
		&\le c_k c_{k,T}^2\int_0^t\d s\int_{-\infty}^\infty\d y\
			|\mathcal{A}(s\,,y)|^2\hskip1in[\text{by \eqref{||u(s,y)||}}]\\
		& = \frac{c_kc_{k,T}^2}{2\pi}\int_0^t\d s\int_{-\infty}^\infty\d\xi\
			\left| \widehat{\mathcal{A}(s)}(\xi)\right|^2
			= \frac{tc_kc_{k,T}^2}{2\pi N}\int_0^t\frac{\d s}{s}\int_{-\infty}^\infty\d\xi\
			\left| \widehat{\mathcal{A}(s)}(t\xi/(Ns))\right|^2,
	\end{split}\end{equation}
	owing to Plancherel's theorem and a change of variables. The correct change of variables is
	slightly tricky to find. But once we have it set up, as we have done above, we note that
	\begin{align*}
		\widehat{\mathcal{A}(s)}(t\xi/(Ns)) &= \frac1N\int_0^N\left[
			\exp\left( i\frac{tx\xi}{N(t+\varepsilon)} -
			\frac{t^2(t+\varepsilon-s)\xi^2}{2s(t+\varepsilon)N^2}\right) -
			\exp\left( i\frac{x\xi}{N} - \frac{t(t-s)\xi^2}{2sN^2}\right)\right]\d x\\
		&= \int_0^1\left[
			\exp\left( i\frac{ty\xi}{t+\varepsilon} - \frac{t^2(t+\varepsilon-s)\xi^2}{2s(t+\varepsilon)N^2}\right) -
			\exp\left( i y\xi - \frac{t(t-s)\xi^2}{2sN^2}\right)\right]\d y\\
		&= J_1 + J_2,
	\end{align*}
	where
	\begin{align*}
		J_1 & := \int_0^1 \e^{ity\xi/(t+\varepsilon)}\,\d y \times
			\left[ \exp\left(- \frac{t^2(t+\varepsilon-s)\xi^2}{2s(t+\varepsilon)N^2}\right) -
			\exp\left( - \frac{t(t-s)\xi^2}{2sN^2}\right)\right],
			\quad\text{and}\\
		J_2 & := \int_0^1\left[
			\exp\left( i\frac{ty\xi}{t+\varepsilon} \right) -
			\exp( iy\xi )\right]\d y\times
			\exp\left( - \frac{t(t-s)\xi^2}{2sN^2}\right).
	\end{align*}
	Since $(a+b)^2\le 2a^2+2b^2$ for all $a,b\in\R$,  we see from \eqref{pre:T_A} that
	\begin{equation}\label{T_A}
		T_{\mathcal{A}}^2 \le \frac{2tc_kc_{k,T}^2}{2\pi N}\int_0^t \frac {\d s}s \int_{-\infty}^\infty\d\xi\
		|J_1|^2 + \frac{2tc_kc_{k,T}^2}{2\pi N}\int_0^t   \frac {\d s}s\int_{-\infty}^\infty\d\xi\
		|J_2|^2.
	\end{equation}
	Define,
	\begin{equation}\label{varphi}
		\varphi(z) := \frac{1-\cos z}{z^2}\qquad\text{for all $z\in\R\setminus\{0\}$},
	\end{equation}
	and $\varphi(0)=1/2$ to preserve continuity. It is then easy to see that
	\begin{align*}
		|J_1| &= \sqrt{2\varphi\left( \frac{t\xi}{t+\varepsilon}\right)}\,
			\left| \exp\left(- \frac{t^2(t+\varepsilon-s)\xi^2}{2s(t+\varepsilon)N^2}\right) -
			\exp\left( - \frac{t(t-s)\xi^2}{2sN^2}\right)\right|\\
		&=\sqrt{2\varphi\left( \frac{t\xi}{t+\varepsilon}\right)}\,
			\exp\left( - \frac{t(t-s)\xi^2}{2sN^2}\right)
			\left| 1 - {\exp\left( -\frac{\varepsilon t\xi^2}{
			2(t+\varepsilon)N^2}\right)} \right|
	\end{align*}
	Therefore,
	\begin{align}
		\int_0^t\frac{\d s}{s}\int_{-\infty}^\infty\d\xi\ |J_1|^2
			&\le 2\int_0^t\frac{\d s}{s}\int_{-\infty}^\infty\d\xi\
			\varphi\left( \frac{t\xi}{t+\varepsilon}\right)\,
			\exp\left( - \frac{t(t-s)\xi^2}{sN^2}\right)
			\left| 1 - {\exp\left( -\frac{\varepsilon t\xi^2}{
			2(t+\varepsilon)N^2}\right)} \right|^2 \notag \\
				&\le C\int_0^t\frac{\d s}{s}\int_{-\infty}^\infty\d\xi\
				\frac{1}{\xi^2}
				\exp\left( - \frac{t(t-s)\xi^2}{sN^2}\right)
				\left| 1 - {\exp\left( -\frac{\varepsilon t\xi^2}{
				2(t+\varepsilon)N^2}\right)} \right|^2 \notag \\
				&\le \frac{C}{N}\int_1^\infty \frac{\d r}{r}\int_{-\infty}^\infty\d z\
				\frac{1}{z^2}
				\exp\left( - t(r-1)z^2 \right)
				\left| 1 - {\exp\left( -\frac{\varepsilon t z^2}{
				2(t+\varepsilon)}\right)} \right|^2 \notag \\
				&\le \frac{C}{N}\int_1^\infty \frac{\d r}{r}\int_{-\infty}^\infty\d z\
				\frac{1}{z^2}
				\exp\left( - t(r-1)z^2 \right)
				\frac{\varepsilon t z^2}{ 2(t+\varepsilon)} \notag \\
				&\le \frac{C\varepsilon}{2N}\int_1^\infty \frac{\d r}{r}\int_{-\infty}^\infty\d z\
				\exp\left( - t(r-1)z^2 \right) \notag \\
				& = \frac{C\varepsilon}{N}\int_1^\infty \frac{1}{r \sqrt{t(r-1)}} \d r \notag \\
				& = \frac{C\varepsilon}{N\sqrt{t}},
				\label{J1}
		\end{align}
		where in the third step we have changed the variables $z=\xi/N$ and $r=t/s$,
		in the fourth step we have applied the inequality $(1-e^{-x^2})^2 \le
		1-e^{-x^2}\le x^2$, and the constant $C$ is a generic constant that may
		change values at each appearance .

	Next, we estimate the same quantity but where $J_1$ is replaced by $J_2$.
	A few lines of computation show that
	\[
		\int_0^1\left[
		\exp\left( i\frac{ty\xi}{t+\varepsilon} \right) -
		\exp( iy\xi )\right]\d y
		=\frac{\e^{i\xi}}{i\xi}\left[\exp\left( \frac{-i\varepsilon\xi}{t+\varepsilon}\right)-1\right]
		+ \frac{\varepsilon}{it\xi}\left[\exp\left( \frac{it\xi}{t+\varepsilon}\right) - 1\right],
	\]
	provided that $\xi\neq 0$. Because $(a+b)^2\le 2a^2+2b^2$ for all $a,b\in\R$,
	\begin{align*}
		\left| \int_0^1\left[
			\exp\left( i\frac{ty\xi}{t+\varepsilon} \right) -
			\exp( iy\xi )\right]\d y \right|^2
			&\le \frac{4}{\xi^2}\left[ 1 - \cos\left( \frac{\varepsilon\xi}{t+\varepsilon}\right)\right]
			+ \frac{4\varepsilon^2}{t^2\xi^2}
			\left[ 1 - \cos\left( \frac{t\xi}{t+\varepsilon}\right)\right]\\
		&\le\frac{2}{\xi^2}\left( \frac{\varepsilon\xi}{t+\varepsilon}\right)^2
			+ \frac{2\varepsilon^2}{t^2\xi^2}
			\left( \frac{t\xi}{t+\varepsilon}\right)^2
			< \frac{4\varepsilon^2}{t^2+\varepsilon^2},
	\end{align*}
	since $1-\cos\theta\le\frac12\theta^2$ for all $\theta\in\R$. Alternatively,
	we could have used the tautological bound, $1-\cos\theta\le2$ in order to deduce
	\[
		\left| \int_0^1\left[
		\exp\left( i\frac{ty\xi}{t+\varepsilon} \right) -
		\exp( iy\xi )\right]\d y \right|^2
		\le \frac{8}{\xi^2} + \frac{8\varepsilon^2}{t^2\xi^2}
		\le\frac{8}{\xi^2}\left( \frac{t^2 + \varepsilon^2}{t^2}\right).
	\]
	Combine the preceding two bounds in order to see that
	\[
		\left| \int_0^1\left[
		\exp\left( i\frac{ty\xi}{t+\varepsilon} \right) -
		\exp( iy\xi )\right]\d y \right|^2 \le
		8\left\{\left(\frac{\varepsilon^2}{t^2+\varepsilon^2}\right)\wedge
		\left( \frac{t^2 + \varepsilon^2}{t^2\xi^2}\right)\right\}.
	\]
	Consequently,
	\begin{align*}
		\int_0^t\frac{\d s}{s}\int_{-\infty}^\infty\d\xi\ |J_2|^2
			&\le 8\int_0^t\frac{\d s}{s}\int_{-\infty}^\infty\d\xi\
			\exp\left( - \frac{t(t-s)\xi^2}{sN^2}\right)
			\left[ \left( \frac{\varepsilon^2}{t^2 + \varepsilon^2} \right)
			\wedge \left( \frac{t^2 + \varepsilon^2}{t^2\xi^2}\right)\right]\\
		&=\frac{8\log N}{t}\int_{-\infty}^\infty G_{N,t}(\xi)
			\left[ \left( \frac{\varepsilon^2}{t^2 + \varepsilon^2} \right)
			\wedge \left( \frac{t^2 + \varepsilon^2}{t^2\xi^2}\right)\right]\d\xi,
	\end{align*}
	where $G_{N,t}$ is defined in \eqref{G_{N,t}} in the Appendix.
	Lemma \ref{G_N:limit} of the Appendix now tells us that
	\begin{equation}\label{J2}\begin{split}
		&\int_0^t\frac{\d s}{s}\int_{-\infty}^\infty\d\xi\ |J_2|^2\\
		&\hskip.3in\le  56 \log (N)\log_+(1/t)\int_{-\infty}^\infty
			\left[ \left( \frac{\varepsilon^2}{t^2 + \varepsilon^2} \right)
			\wedge \left( \frac{t^2 + \varepsilon^2}{t^2\xi^2}\right)\right]
			\log_+(1/|\xi|)\, \d\xi\\
		&\hskip.3in= 56\log (N)\log_+(1/t)\left( \frac{t^2 + \varepsilon^2}{t^2}\right)\int_{-\infty}^\infty
			\left[ \left( \frac{\varepsilon^2t^2}{(t^2 + \varepsilon^2)^2} \right)
			\wedge \frac{1}{\xi^2}\right]
			\log_+(1/|\xi|)\,\d\xi\\
		&\hskip.3in< \frac{560\log (N) \log_+(1/t)\varepsilon}{t};
	\end{split}\end{equation}
	see Lemma \ref{lem:easy:int} in the Appendix. Combine \eqref{T_A} with
	\eqref{J1} and \eqref{J2} in order to find that
	\[
		T_{\mathcal{A}}^2 \le  a_{T,k,\delta}\,\frac{\log
		N}{N}\frac{\varepsilon}{t^{1+\delta}},
	\]
	where $a_{T,k, \delta}$ is a real number depends only on $(T,k, \delta)$.
	We combine this bound with \eqref{T_B} and then \eqref{STAB} to conclude the proof.
\end{proof}

  We are now ready for the following.

 \begin{proof}[Proof of Proposition \ref{pr:tightness}]
	We assume without incurring loss in generality that $T>1/\e$.  Choose and fix
	two arbitrary numbers $\alpha\in(0\,,1)$ and $\beta\in(0\,,1)$.  On one hand,
	Lemma \ref{lem:tightness} implies that, uniformly for all
	$\varepsilon\in(0\,,1/\e)$, $N\ge\e$, and $t\in(\varepsilon^\beta,T]$,
	\begin{equation}\label{S-S1}
		\left\| \mathcal{S}_{N,t+\varepsilon}-\mathcal{S}_{N,t}
		\right\|_k \le M\varepsilon^{(1-2\beta(1+\delta))/2}
		\sqrt{\frac{\log N}{N}},
	\end{equation}
	with $M:=K^{1/k}$. [The condition $T>1/\e$ is there merely to ensure that
	$(\varepsilon^\beta,T]\neq\varnothing$].  On the other hand, Lemma
	\ref{lem:||S||} implies the existence of a real number $M'= M'(T,k\,,\alpha)$
	such that, uniformly for all $N\ge\e$ and $t\in(0\,,\varepsilon^\beta]$,
	\begin{equation}\label{S-S2}
		\left\| \mathcal{S}_{N,t+\varepsilon}-\mathcal{S}_{N,t}
		\right\|_k \le \left\| \mathcal{S}_{N,t+\varepsilon}\right\|_k
		+ \left\| \mathcal{S}_{N,t} \right\|_k \le M'\varepsilon^{\beta\alpha/2}\sqrt{\frac{\log N}{N}}.
	\end{equation}
	Choose $\beta = (2+\alpha+2\delta)^{-1}$ to match the exponents of
	$\varepsilon$ in \eqref{S-S1} and \eqref{S-S2} and hence conclude the
	asserted inequality of the proposition with $L:=M\vee M'$ and
	$\gamma:=\alpha/\{2(2+\alpha+2\delta)\}$. To finish the proof we note that
	$\gamma$ can be any number in $(0\,,1/6)$ since $\alpha\in(0\,,1)$ and
	$\delta>0$ are arbitrary.
 \end{proof}

Armed with Proposition \ref{pr:tightness}, we  conclude the section with the
following.

\begin{proof}[Proof of Theorem \ref{th:FCLT}]
	Choose and fix some $T>0$.  By Lemma \ref{lem:||S||} and Proposition
	\ref{pr:tightness}, a standard application of Kolmogorov's continuity theorem
	and the Arzel\`{a}-Ascoli theorem  ensures that $\{
	\sqrt{N/\log(N)}\,\mathcal{S}_{N,\bullet} \}_{N\ge\e}$ is a tight net of
	processes on $C[0\,,T]$. Therefore, it remains to prove that the
	finite-dimensional distributions of the process $t\mapsto\sqrt{N/\log
	N}\,\mathcal{S}_{N,t}$ converge to those of $\sqrt{2}B$; see for example
	Billingsley \cite{Bil99}.

	Let us choose and fix  some $T>0$ and $m\ge 1$ points $t_1,\ldots,t_m\in(0\,,T)$.
	Proposition \ref{pr:Cov:asymp} ensures that, for every $i,j=1,\ldots,m$,
	\begin{equation}\label{C}
		\Cov\left( \mathcal{S}_{N,t_i}\,,\mathcal{S}_{N,t_j}\right)
		\sim 2(t_i\wedge t_j)\frac{\log N}{N}\qquad\text{as $N\to\infty$}.
	\end{equation}
	Therefore, there exists $N_0>0$ such that
	\begin{equation}\label{Var(S)>}
		\Var(\mathcal{S}_{N,t_i}) \ge  t_i\,\frac{\log N}{N}
		\qquad\text{for every $i=1,\ldots,m$ and $N>N_0$. }
	\end{equation}
	Choose and fix an arbitrary $N>N_0$,
	and consider the following random variables:
	\[
		F_i  :=  \frac{\mathcal{S}_{N,t_i}}{\sqrt{\Var(\mathcal{S}_{N,t_i})}}
		\qquad \text{for }i=1,\ldots,m,
	\]
	and define $C_{i,j} := \Cov( F_i\,,F_j)$ for every $i,j=1,\ldots,m.$
	We will write $F:=(F_1\,,\ldots,F_m)$,
	and let $G=(G_1\,,\ldots,G_m)$ denote a centered Gaussian random vector
	with covariance matrix $C=(C_{i,j})_{1\le i,j\le m}$.

	Recall from \eqref{g+v} the random fields $v_{N,t_1},\ldots,v_{N,t_m}$, and define
	rescaled random fields $V_1,\ldots,V_m$ as follows:
	\[
		V_i := \frac{v_{N,t_i}}{\sqrt{\Var(\mathcal{S}_{N,t_i})}}\qquad
		\text{for }i=1,\ldots,m.
	\]
	According to \eqref{S=delta(v)},
	$F_i =\delta( V_i)$ for all $i=1,\ldots,m$. Lemma \ref{lem:E<DS,v>} ensures
	that $\E\<DF_i\,,V_j\>_\HH=C_{i,j}$ for all $i,j=1,\ldots,m$. Therefore,
	Lemma \ref{lemma: NP 6.1.2} ensures that
	\[
		\left| \E h(F) - \E h(G) \right| \le\tfrac12 \|h''\|_\infty
		\sqrt{\sum_{i,j=1}^m \Var\< DF_i\,,V_j\>_\HH},
	\]
	for all $h\in C^2_b(\R^m)$.
	Proposition \ref{Var<Ds,v>} and \eqref{Var(S)>} together assure us that
	\[
		\Var\<DF_i\,,V_j\>_\HH =
		\frac{\Var\<D\mathcal{S}_{N,t_i}~,~v_{N,t_j}\>_\HH}{%
		\Var(\mathcal{S}_{N,t_i})\Var(\mathcal{S}_{N,t_j})}
		\le \frac{K_T\log N}{N\min_{1\le k\le m}t_k}.
	\]
	whence
	\begin{equation}\label{h(F)}
		\left| \E h(F) - \E h(G) \right| \le c \|h''\|_\infty\sqrt{\log N}/\sqrt N,
	\end{equation}
	for $c=\frac12\sqrt{K_T/\min_{1\le k\le m}t_k}$.

	Now we let $N\to\infty$: Thanks to \eqref{C},
	$C_{i,j}\to (t_i\wedge t_j)/\sqrt{t_it_j}$ whence
	$G$ converges weakly to $(B_{t_i}/\sqrt{t_i})_{1\le i\le m}$ as $N\to\infty$.
	Therefore, it follows from \eqref{h(F)} that $F$ converges weakly to
	$(B_{t_i}/\sqrt{t_i})_{1\le i\le m}$ as $N\to\infty$. One more appeal to \eqref{C}
	shows that
	\[
		\sqrt{\frac{N}{\log N}}
		\left( \frac{\mathcal{S}_{N,t_1}}{\sqrt{2t_1}}\,,\ldots,
		\frac{\mathcal{S}_{N,t_m}}{\sqrt{2t_m}}\right)
		\xrightarrow{\rm d\,}
		\left( \frac{B_{t_1}}{\sqrt{t_1}}\,,\ldots,\frac{B_{t_m}}{\sqrt{t_m}}\right)
		\qquad\text{as $N\to\infty$}.
	\]
	It follows from this fact that the finite-dimensional distributions of
	$t\mapsto\sqrt{N/\log N}\,\mathcal{S}_{N,t}$ converge to those of
	$\sqrt 2\,B$ as $N\to\infty$. This verifies the remaining goal of this proof.
\end{proof}

\appendix
\section{Appendix}

We include in this section a few technical results that have been used along
the paper.  In order to describe the first result, define
\begin{align}\label{G_{N,t}}
	G_{N, t}(x) := \frac{t}{\log N}\int_0^t \exp\left( - \frac{(t - s)t}{s}\cdot \frac{x^2}{N^2}\right)
	\,\frac{\d s}{s}\qquad
	\text{for all $N,t>0$ and $x\in\R\setminus\{0\}$}.
\end{align}

\begin{lemma}\label{G_N:limit}
	For every $t>0$ and $x \in \R\setminus\{0\}$,
	\[
		\sup_{N \geq\e}G_{N, t}(x)
		\le 7t\log_+(1/t) \log_+(1/|x|),
	\]
	where we recall that $\log_+(w):=\log(\e + w)$ for all $w\ge0$. Moreover,
	\begin{align} \label{G_Nlimit}
		\lim_{N\to\infty}G_{N, t}(x) = 2t
		\qquad\text{for every $t>0$ and $x\in\R$.}
	\end{align}
\end{lemma}

\begin{proof}
	We change variables in order to see that
	\[
		G_{N, t}(x) =
		\frac{t}{\log N} \int_0^{\infty}
		\frac{\e^{-s}}{s + \frac{tx^2}{N^2}}\,\d s
		= \frac{t}{\log N}(A_N - B_N+C_N).
	\]
	where
	\[
		A_N := \int_0^{1}\frac{\d s}{s + \frac{tx^2}{N^2}}= \log\left( \frac{N^2}{tx^2}+1\right),\quad
		B_N :=\int_0^1\frac{1-\e^{-s}}{s + \frac{tx^2}{N^2}}\,\d s,\quad
		C_N := \int_1^\infty\frac{\e^{-s}}{s + \frac{tx^2}{N^2}}\,\d s.
	\]
	This proves \eqref{G_Nlimit} because $B_N,C_N\in(0\,,1)$. Next,  we observe that
	\[
		\frac{N^2}{tx^2}+1
		\le N^2\left( \e + t^{-1}\right)\left( \e + |x|^{-2}\right),
	\]
	whence
	\[
		A_N\le 2\log N+\log_+(1/t)  + 2\log_+(1/|x|)\le
		5\log (N)\log_+(1/t)\log_+(1/|x|),
	\]
	for all $N\ge\e$, $t>0$, and all non-zero $x$.
	This does the job since $B_N+C_N\le 2$,
	which is manifestly less than or equal to $2\log_+(1/t)\log_+(1/|x|)$.
\end{proof}

The following lemma provides a useful heat-kernel formula.

\begin{lemma} \label{lem7.2}
	For all $N,t>0$, we have
	\[
		\int_0^N\d x_1\int_0^N\d x_2\  \bm{p}_t (x_1-x_2)
		=\frac {N} {\pi} \int_{-\infty}^\infty \varphi(z)
		\e^{ - t z^2/(2N^2)}\, \d z.
	\]
	where $\varphi(z)$ was defined in \eqref{varphi}.
\end{lemma}

\begin{proof}
	Plancherel's theorem implies that
	\begin{align*}
		\int_0^N\d x_1\int_0^N\d x_2\
			\bm{p}_t (x_1-x_2) &=
			\frac 1{ 2\pi} \int_{-\infty}^\infty  | \widehat{\mathbf{1}_{[0,N]}} (y) |^2
			\e^{ -ty^2/2}\, \d y \\
		&= \frac {N^2} { 2\pi} \int_{-\infty}^\infty  |\widehat{\mathbf{1}_{[0,1]}} (Ny) |^2
			\e^{ -ty^2/2}\, \d y.
	\end{align*}
	A change of variables  $[z=Ny]$ implies the lemma,
	since $|\widehat{\mathbf{1}_{[0,1]}} (z) |^2 = 2\varphi(z)$ for all $z\in\R$.
\end{proof}

 Finally, we mention the following simple inequality.

 \begin{lemma}\label{lem:easy:int}
 	For every $\varepsilon\in(0\,,1)$,
 	\[
		\int_{-\infty}^\infty
		\left( \varepsilon \wedge \frac{1}{z^2}\right)
		\log_+(1/|z|)\, \d z  < 10\sqrt\varepsilon.
	\]
 \end{lemma}

 \begin{proof}
 	Let $J(\varepsilon)$ denote the integral in question.
	Because $\varepsilon<1$ and $\log (2\e) \le 2$,
 	\[
 		J(\varepsilon) = 4\int_{1/\e}^\infty
		\left( \varepsilon \wedge \frac{1}{z^2}\right) \d z
		+2\varepsilon\int_0^{1/\e}\log(1/z)\,\d z
		< 4\varepsilon\int_{1/\e}^\infty
		\left( 1 \wedge \frac{1}{\varepsilon z^2}\right) \d z + 2\varepsilon,
	\]
	since $z\mapsto\log(1/z)$ defines a probability density
	function on $(0\,,1)$ and $0<\varepsilon<1$. Change variables to see that
 	\[
 		J(\varepsilon) < 4\sqrt\varepsilon\int_{\sqrt\varepsilon/\e}^\infty
		\left( 1 \wedge \frac{1}{r^2}\right) \d r + 2\varepsilon
		= 8\sqrt\varepsilon + 2\left( 1-\frac{2}{\e}\right)\varepsilon,
	\]
	which readily implies the result since $\varepsilon<\sqrt\varepsilon$.
 \end{proof}

\begin{lemma}\label{lem:||U-1||}
	Let $c_{T,k}$ be the constant defined in \eqref{||u(s,y)||}
	and set $C_T:=\pi^{1/4}2^{-1/2}\:c_{T,2}$ . Then,
	\[
		\sup_{x\in\R}\|U(t\,,x)-1\|_2\le
		C_T t^{1/4}\qquad\text{for all $t\in(0\,,T]$}.
	\]
\end{lemma}

\begin{proof}
	Owing to \eqref{E:U}, $\E [U(t\,,x)]=1$ for all $t\in(0\,, T ]$ and $x\in\R$,
	and
	\begin{align*}
		\Var [U(t\,,x)] &=\int_0^t\d s\int_{-\infty}^\infty\d y\
			\left| \bm{p}_{s(t-s)/t}\left(y- \frac{s}{t}x \right)\right|^2
			\E\left(|U(s\,,y)|^2\right)\\
		&\le c_{T,2}^2\int_0^t\d s\int_{-\infty}^\infty\d y\
			\left| \bm{p}_{s(t-s)/t}\left(y- \frac{s}{t}x \right)\right|^2
			&\text{[see \eqref{U} and \eqref{||u(s,y)||}]}\\
		&= c_{T,2}^2\int_0^t \bm{p}_{2s(t-s)/t}(0)\d s =
		  	c_{T,2}^2\sqrt{\pi t/4},
	\end{align*}
	thanks to the semigroup property of the heat kernel and a few computations.
	This completes the proof.
\end{proof}

\noindent
{\bf Acknowledgement.} We would like to thank the referees for their valuable
and useful comments. 
D. Khoshnevisan is supported in part by  NSF grants  DMS-1855439. D. Nualart is supported in part by  NSF grants DMS-1811181.  
F. Pu is grateful to University of Utah where the work was carried out.

\end{document}